\documentclass[10pt,a4paper]{article}

\usepackage[latin1]{inputenc}
\usepackage{amsmath, amsfonts, amssymb}
\usepackage[margin =2.5cm]{geometry}
\usepackage{graphicx}
\usepackage{enumitem}
\usepackage{hyperref}
\hypersetup{
	colorlinks =true,        
	linkcolor =blue,         
	citecolor =blue,         
	urlcolor =blue           
}

\usepackage{amsthm}
\usepackage{thmtools}
\usepackage[capitalise]{cleveref}
\usepackage{xr}

\newtheorem{theorem}{Theorem}[section]
\newtheorem{proposition}{Proposition}[section]
\newtheorem{corollary}{Corollary}[section]
\newtheorem{definition}{Definition}[section]

\newtheorem{lemma}{Lemma}[section]
\declaretheorem[style =remark,qed =$\Diamond$,Refname ={Remark,Remarks},within=section]{remark}
\declaretheorem[style =remark,Refname ={Example,Examples},within=section]{example}

\allowdisplaybreaks

\newcommand{\setto}{\rightrightarrows} 
\newcommand{\dom}{\operatorname{dom}}
\newcommand{\inte}{\operatorname{int}}
\newcommand{\range}{\operatorname{range}}
\DeclareMathOperator*{\argmin}{arg\,min}

\newcommand{\Id}{\operatorname{Id}}
\newcommand{\Fix}{\operatorname{Fix}}

\newcommand{\Sing}{\operatorname{Sing}}
\newcommand{\dist}{\operatorname{d}}

\newcommand{\prox}{\operatorname{prox}}
\DeclareMathOperator*{\Limsup}{\operatorname{Limsup}}

\newcommand{\StrFix}{\mathbf{Fix}\,}

\usepackage[most]{tcolorbox}
\usepackage[normalem]{ulem}

\begin{document}
\title{Union Averaged Operators with Applications to Proximal Algorithms for Min-Convex Functions}

\author{
	Minh N.\ Dao\thanks{CARMA, 
		University of Newcastle, 
		Callaghan, NSW 2308, Australia.
		E-mail:~\href{mailto:daonminh@gmail.com}{daonminh@gmail.com}}
	\and
	Matthew K.\ Tam\thanks{Institut f\"ur Numerische und Angewandte Mathematik,
		Universit\"at G\"ottingen,
		37083 G\"ottingen, Germany.
		\mbox{E-mail:}~\href{mailto:m.tam@math.uni-goettingen.de}{m.tam@math.uni-goettingen.de}}
}

\date{October 17, 2018}

\maketitle

\begin{abstract}
In this paper, we introduce and study a class of structured set-valued operators which we call \emph{union averaged nonexpansive}. At each point in their domain, the value of such an operator can be expressed as a finite union of single-valued averaged nonexpansive operators. We investigate various structural properties of the class and show, in particular, that is closed under taking unions, convex combinations, and compositions, and that their fixed point iterations are locally convergent around strong fixed points. We then systematically apply our results to analyze proximal algorithms in situations where union averaged nonexpansive operators naturally arise. In particular, we consider the problem of minimizing the sum two functions where the first is convex and the second can be expressed as the minimum of finitely many convex functions.
\end{abstract}

\paragraph{Mathematics Subject Classification (MSC 2010):}
90C26 $\cdot$ 
47H10 $\cdot$ 
47H04 
    

\section{Introduction}
The notion of an \emph{averaged nonexpansive} operator is one which nicely balances two properties of importance in the context of fixed point algorithms, namely, \emph{usefulness} and \emph{applicability}.  In this context, usefulness is meant in the sense that algorithms based on such operators are provably convergent, for instance, by appealing to Opial-type results \cite{Ceg12,Opi67}, and applicability is meant in sense that the class of averaged nonexpansive is significantly rich so as to include many commonly encountered operators. Indeed, the class includes all \emph{firmly nonexpansive} operators as well as their convex combinations and compositions \cite[Section~4.5]{BauCom17}. For further information on averaged operators, the reader is referred to \cite{BaiBruRei78,BauCom17,Ceg12} and the references therein.

In many applications, particularly those involving some of kind of nonconvexity, the involved algorithmic operator is not averaged nonexpansive. Nevertheless, it is sometimes still the case that some underlying averaged nonexpansive structure which can be exploited is present. A notable example is provided by \emph{sparsity constrained} optimization in which the feasible region is a lower-level set of the $\ell_0$-psuedo norm. This set can be naturally expressed as the union of a finite number of ``sparsity subspaces''. Consequently, at each point in the space, its metric projector can be expressed as the union a subset of the averaged nonexpansive projectors onto these subspaces, although the projector onto the lower-set itself is not averaged nonexpansive. Indeed, this type of decomposability was consider in \cite{Tam17} which was, in turn, inspired by \cite{BauNol14}.

In this work, we aim to exploit structure of the aforementioned type. More precisely, we consider a class of set-valued operators which we call \emph{union averaged nonexpansive} (as well as the notion of a \emph{union nonexpansive operator}). At each point in the ambient space, the value of these operators can be described as a union of single-valued averaged nonexpansive operators from a finite family. A related notion, \emph{union paracontracting} operators, was previously studied by the second author in \cite{Tam17}. A significant short-coming of the class of union paracontracting operators is that they are, in general, not closed under operators taking convex combination and compositions, thus making it more difficult to determine if a given operators belong to the class. The situation is remedied by union average nonexpansiveness.

The remainder of this paper is organized as follows. We begin, in Section~\ref{s:preliminaries}, by recalling the necessary mathematical preliminaries. In Section~\ref{s:union averaged operators} we introduce the notion of \emph{union averaged nonexpansive} operators and study their closure and fixed point properties (Proposition~\ref{prop:closedness}~\&~\ref{prop:fixed points}). In Section~\ref{s:convergence}, we provide a variation of \cite{Els92} (Theorem~\ref{thm:KM iterations with admissible control}) and, as a corollary, deduce local convergence for union (averaged) nonexpansive maps around their strong fixed points (Theorem~\ref{thm:union}). In Section~\ref{s:min-convex}, we introduce and study functions which can be expressed as the minimum of finitely many convex functions; we term these functions \emph{min-convex} before concluding, in Section~\ref{s:proximal algorithms} with a systematic study of proximal algorithms applied to minimization of min-convex problems including various \emph{projection algorithms}, the \emph{proximal point algorithm}, \emph{forward-backward splitting} and \emph{Douglas--Rachford splitting}.

\section{Preliminaries}\label{s:preliminaries}
In this section, we recall the necessary preliminaries for the subsequent sections. Unless stated otherwise, throughout this work we assume that
\begin{equation*}
\fcolorbox{gray!30}{gray!30}{  X\text{ is a (real) finite-dimensional Hilbert space}  }
\end{equation*}
with inner-product $\langle\cdot,\cdot\rangle$ and induced norm $\|\cdot\|$.

In order to introduce the two new classes of structured set-valued operators in Definition~\ref{def:union averaged}, we first recall the definitions and basic properties of their single-valued counterparts. The term ``averaged'' originally appeared in \cite{BaiBruRei78}.

\begin{definition}[Averaged nonexpansive operators]
\label{def:averaged}
A single-valued operator $T\colon X\to X$ is said to be \emph{nonexpansive} if 
\begin{equation*}
\forall x,y\in X,\quad \|Tx-Ty\|\leq \|x-y\|,
\end{equation*}
and \emph{$\alpha$-averaged nonexpansive} if $\alpha\in(0,1)$ and
\begin{equation*}
\forall x,y\in X,\quad \|Tx-Ty\|^2 +\frac{1-\alpha}{\alpha}\|(\Id-T)x-(\Id-T)y\|^2 \leq \|x-y\|^2.
\end{equation*}
We say $T$ is \emph{averaged nonexpansive} if there is an $\alpha\in(0,1)$ such that $T$ is $\alpha$-averaged nonexpansive.
\end{definition}    

It follows immediately from the above definition that every averaged nonexpansive operator is nonexpansive. 
The precise relationship between the two classes is given in the following proposition.
\begin{proposition}[Characterizations of averaged nonexpansiveness]
\label{prop:AN characterizations}
Let $T\colon X\to X$ be an operator and let $\alpha\in(0,1)$. The following assertions are equivalent.
\begin{enumerate}[label =(\alph*)]
\item $T$ is $\alpha$-averaged nonexpansive.
\item $T =(1-\alpha)\Id+\alpha R$ for some nonexpansive operator $R\colon X\to X$.
\item $(1-1/\alpha)\Id+(1/\alpha)T$ is nonexpansive.
\end{enumerate}
\end{proposition}
\begin{proof}
Follows by combining \cite[Definition~4.33 and Proposition~4.35]{BauCom17}.
\end{proof}

The following proposition shows that the classes of nonexpansive and averaged nonexpansive operators are both closed under taking convex combination and under compositions. Such properties are of interest because they provide a way to verify that a given operator is averaged nonexpansive in the case that it can be represented in terms of simpler operators whose averaged nonexpansiveness can be more easily checked.

\begin{proposition}[Convex combinations and compositions]\label{prop:AN closedness}
Let $J :=\{1,\dots,m\}$ and let $T_j\colon X\to X$ be $\alpha_j$-averaged nonexpansive (resp. nonexpansive) for each $j\in J$. Then the following assertions hold.
\begin{enumerate}[label =(\alph*)]
\item\label{it:AN convex combination} $\sum_{j\in J}\omega_jT_j$ is $\alpha$-averaged nonexpansive with
\begin{equation*}
\alpha :=\sum_{j\in J}\omega_j\alpha_j
\end{equation*}		
(resp. nonexpansive) whenever $(\omega_j)_{j\in J}\subseteq\mathbb{R}_{++}$ with $\sum_{j\in J}\omega_j =1$.
\item\label{it:AN compositions alpha} $T_m\circ\dots\circ T_2\circ T_1$ is $\alpha$-averaged nonexpansive with
\begin{equation*}
\alpha :=\left(1+\left(\sum_{j\in J}\frac{\alpha_j}{1-\alpha_j}\right)^{-1}\right)^{-1}
\end{equation*}
(resp. nonexpansive).
\end{enumerate}
\end{proposition}	
\begin{proof}
See \cite[Propositions~2.2~\&~2.5]{ComYam15}. Note also that Definition~\ref{def:averaged} for averaged operators appeared in \cite[Definition~2.2.14]{Ceg12} under the name $\nu$-firmly nonexapnsive operators with $\nu$ taken to be $(1-\alpha)/\alpha$. Appealing to the connection given in \cite[Corollary~2.2.17]{Ceg12}, the conclusion can be also deduced from \cite[Theorems~2.2.35~\&~2.2.42]{Ceg12}.
\end{proof}

The second notion that will be used in Definition~\ref{def:union averaged} is that of an \emph{outer semicontinuous} map \cite[Section~3B]{DonRoc14}. In what follows, we recall its definition and basic properties.

\begin{definition}[Outer semicontinuity]
Let $X$ and $Y$ be Hilbert spaces. A set-valued map $\phi\colon X\setto Y$ is \emph{outer semicontinuous (osc)} at $\bar{x}$ if
\begin{equation*}
\phi(\bar{x}) \supseteq \Limsup_{x\to\bar{x}}\phi(x) :=\{y\in Y:\exists\,x_n\to \bar{x},\,\exists\,y_n\to y\text{ with }y_n\in \phi(x_n)\}. 
\end{equation*}  
That is, the limit supremum is understood in the sense of the Painlev\'e--Kuratowski outer limit on $X\times Y$.	
\end{definition}

\begin{proposition}[Cartesian products]\label{prop:osc_product}
Let $J :=\{1,\dots,m\}$ and let $\phi_j\colon X\setto Y_j$ be osc for each $j\in J$. Then the mapping
$\phi\colon X\setto Y_1\times\dots\times Y_m$  defined by
$$x\mapsto \phi(x) :=\phi_1(x)\times\dots\times\phi_m(x)$$
is also osc.
\end{proposition}
\begin{proof}
Let $\bar{x}\in X$ and consider sequences $x_n\to \bar{x}$ and $y_n :=(y_{1n}, \dots, y_{mn})\to y :=(y_1, \dots, y_m)$ and $y_n\in \phi(x_n)$, or equivalently, $y_{jn}\in \phi_j(x_n)$ for every $j\in J$. Since each $\phi_j$ is osc, it holds that $y_j\in \phi_j(\bar{x})$, and hence $y\in \phi(\bar{x}) =\phi_1(\bar{x})\times\dots\times\phi_m(\bar{x})$.
\end{proof}

\begin{proposition}\label{prop:osc_composition}
Let $\phi\colon X\setto Y$ be osc, $I$ be a nonempty finite index set, $\{T_i\}_{i\in I}$ be a collection of continuous single-valued operators on $X$, and $\varphi\colon X\setto I$ be osc. Then the mapping $\psi\colon X\setto Y$ is osc where $\psi$ is defined by 
\begin{equation*}
x\mapsto \psi(x) :=\{\phi(T_i(x)): i\in \varphi(x)\}.
\end{equation*} 
Consequently, if $T\colon X\to X$ is continuous, then $\phi\circ T$ is osc.
\end{proposition}
\begin{proof}
Let $\bar{x}\in X$ and consider sequences $x_n\to \bar{x}$ and $y_n\to y$ with $y_n\in \psi(x_n)$. By definition, $y_n\in \phi(T_{i_n}(x_n))$ for some $i_n\in \varphi(x_n)\subseteq I$. Since $I$ is finite, by passing to a subsequence, we may assume that $i_n =i\in I$ for all $n$. Then the osc of $\varphi$ implies that $i\in \varphi(\overline{x})$, the continuity of $T_i$ implies that $T_i(x_n)\to T_i(\overline{x})$. As $\phi$ is osc and $y_n\in \phi(T_i(x_n))$, it follows that $y\in \phi(T_i(\overline{x}))\subseteq \psi(\overline{x})$. 
\end{proof}

\section{Unions of averaged operators}\label{s:union averaged operators}
In this section, we introduce the classes of operators which are the main object of study in this work and investigate their properties. We begin with their definition.

\begin{definition}[Union averaged nonexpansive operators]\label{def:union averaged}
A set-valued operator $T\colon X\setto X$ is said to be \emph{union $\alpha$-averaged nonexpansive} (resp.\ \emph{union nonexpansive}) if $T$ can be expressed in the form
\begin{equation}\label{eq:def union averaged}
\forall x\in X,\quad T(x) =\{T_i(x): i\in\varphi(x)\},
\end{equation}
where $I$ is a finite index set, $\{T_i\}_{i\in I}$ is a collection of $\alpha$-averaged nonexpansive (resp.\ nonexpansive) operators on $X$, and $\varphi\colon X\setto I$, called an \emph{active selector}, is an osc operator with nonempty values.

As before, we say $T$ is \emph{union averaged nonexpansive} if there is an $\alpha\in(0,1)$ such that $T$ is union $\alpha$-averaged nonexpansive.
\end{definition}

In order to demonstrate as situation in which union averaged nonexpansiveness naturally arises, we state the following example which we shall return to in Section~\ref{s:min-convex}.
\begin{example}[Sparsity projectors]\label{ex:sparsity}
Let $X =\mathbb{R}^n$ and $s\in\{0,1,\dots,n-1\}$. A common approach in \emph{sparsity optimization} involves minimization over the nonconvex sparsity constraint set
\begin{equation*}
C :=\{x\in\mathbb{R}^n:\|x\|_0\leq s\},
\end{equation*}
where $\|\cdot\|_0$ denotes the \emph{$\ell_0$-functional} which counts the number of nonzero entries in a vector. By denoting $\mathcal{I} :=\{I\in 2^{\{1,2,\dots,n\}}: |I| =s\}$, the set $C$ can be naturally expressed as a union of nonempty subspaces as
\begin{equation*}
C =\bigcup_{I\in\mathcal{I}}C_I\text{ where }C_I :=\{x\in\mathbb{R}^n:x_i\neq 0\text{ only if }i\in I\}.
\end{equation*}
In Proposition~\ref{prop:projection operators}\ref{it:proj}, we shall show that \emph{nearest point projector} onto $C$ is union $1/2$-averaged nononexpansive with
$$ P_C(x) :=\{c\in C:\|x-c\| =\dist(c,C)\} =\{P_{C_I}(x): I\in\varphi(x)\}$$
where $\varphi(x) :=\{I\in\mathcal{I}:\min_{i\in I}|x_i|\geq \max_{i\not\in I}|x_i|\}$.
\end{example}

The following proposition is the union averaged nonexpansive analogue of Proposition~\ref{prop:AN characterizations} and offers equivalent characterizations of union averaged nonexpansiveness. In what follows, the sum of two or more sets is understood in the sense of the \emph{Minkowski sum}.
\begin{proposition}[Equivalent characterizations of union averaged nonexpansiveness]
\label{prop:averaged}
Let $T\colon X\setto X$ be a set-valued operator and let $\alpha\in(0,1)$. The following assertions are equivalent.
\begin{enumerate}[label =(\alph*)]
\item $T$ is union $\alpha$-averaged nonexpansive.
\item\label{it:prop averaged 2} $T =(1-\alpha)\Id+\alpha R$ for some union nonexpansive operator $R\colon X\setto X$.
\item $(1-1/\alpha)\Id+(1/\alpha)T$ is union nonexpansive.
\end{enumerate}
\end{proposition}
\begin{proof}
Follows by combining Definition~\ref{def:union averaged} with Proposition~\ref{prop:AN characterizations}.
\end{proof}

A class of operators related to those in Definition~\ref{def:union averaged}, the class of  \emph{union paracontracting operators}, was introduced in \cite{Tam17}. Recall that a single-valued operator $S\colon X\to X$ is \emph{paracontracting} if it is continuous and \emph{strictly quasi-nonexpansive}, that is,
$$ \forall x\in X\setminus \Fix S,\,\forall y\in\Fix S,\quad \|S(x)-y\| < \|x-y\|. $$
An operator $T\colon X\setto X$ of the form \eqref{eq:def union averaged} is \emph{union paracontracting} if $\{T_i\}_{i\in I}$ is instead a collection of paracontracting operators. 

In general, convex combinations and compositions of paracontracting operators need not stay paracontracting except when the individual operators share a common fixed points; see \cite[Theorem ~2.1.26 and Corollary~2.1.29]{Ceg12}. Consequently, the same is true of union paracontracting operators. As the following proposition shows, this shortcomming is rectified by using averaged nonexpansive operators in place of paracontracting.

\begin{proposition}[Unions, combinations and compositions]\label{prop:closedness}
Let $J :=\{1,\dots,m\}$ and let $T_j\colon X\setto X$ be union $\alpha_j$-averaged nonexpansive (resp. union nonexpansive) for each $j\in J$. Then the following assertions hold.
\begin{enumerate}[label =(\alph*)]
\item\label{it:UAN union} $T\colon X\setto X$ defined by $x\mapsto T(x) :=\cup_{j\in J} T_j(x)$ is union $\alpha$-averaged nonexpansive with 
\begin{equation*}
\alpha :=\max_{j\in J} \alpha_j
\end{equation*}
(resp. union nonexpansive).		  
\item\label{it:UAN convex combination} $\sum_{j\in J}\omega_jT_j$ is union $\alpha$-averaged nonexpansive with
\begin{equation}
\label{eq:convex combination alpha}
\alpha :=\sum_{j\in J}\omega_j\alpha_j
\end{equation}		
(resp. union nonexpansive) whenever $(\omega_j)_{j\in J}\subseteq\mathbb{R}_{++}$ with $\sum_{j\in J}\omega_j =1$.
\item\label{it:UAN composition} $T_m\circ\dots\circ T_2\circ T_1$ is union $\alpha$-averaged nonexpansive with
\begin{equation}
\label{eq:composition alpha}
\alpha :=\left(1+\left(\sum_{j\in J}\frac{\alpha_j}{1-\alpha_j}\right)^{-1}\right)^{-1}
\end{equation}
(resp. union nonexpansive).
\end{enumerate}
\end{proposition}
\begin{proof}
First note that as $T_j$ is union $\alpha_j$-averaged nonexpansive (resp. union nonexpansive), by definition, there exists a finite index set $I_j$, an osc map $\varphi_j\colon X\setto I_j$ and a collection of $\alpha_j$-averaged nonexpansive (resp.\ nonexpansive) operators $\{T_{j,i_j}\}_{i_j\in I_j}$ on $X$ such that $T_j$ can be expressed as
\begin{equation*}
\forall x\in X,\quad T_j(x) =\{T_{j,i_j}(x): i_j\in \varphi_j(x)\}.
\end{equation*}

\ref{it:UAN union}:~For all $x\in X$, the definition of $T$ yields
\begin{equation*}
T(x) =\bigcup_{j\in J} T_j(x) =\{T_{j,{i_j}}(x): j\in J,i_j\in\varphi_j(x)\} =\{T_{j,i}(x): (j,i)\in \varphi(x)\}
\end{equation*}
where $\varphi\colon X\setto J\times \cup_{j\in J} I_j$ is defined as 
\begin{equation*}
\varphi(x) :=\{(j,i): j\in J, i\in \varphi_j(x)\}.
\end{equation*}
Note that $T_{j,i}$ is $\alpha$-averaged nonexpansive (resp.\ nonexpansive) for each $j\in J$ and each $i\in I_j$ since $\alpha\geq \alpha_j$, hence we only need to prove osc of $\varphi$. To this end, consider $x_n\to x$ in $X$ and $(j_n,i_n)\to (j,i)$ in $J\times \cup_{j\in J} I_j$ with $(j_n,i_n)\in \varphi(x_n)$. Since $J$ is finite and $j_n\to j$, there exists $n_0\in \mathbb{N}$ such that $j_n =j$ for all $n\geq n_0$, and we therefore have $i_n\in\varphi_{j_n}(x) =\varphi_j(x)$ for $n\geq n_0$. Outer semicontinuity of $\varphi_j$ implies that $i\in \varphi_j(x)$ and hence that $(j,i)\in \varphi(x)$.  Consequently, $\varphi$ is osc, and $T$ is thus union $\alpha$-averaged nonexpansive (resp. union nonexpansive).

\ref{it:UAN convex combination}:~The definition of the Minkowski sum gives that
\begin{equation*}
\sum_{j\in J} \omega_jT_j(x)
 =\sum_{j\in J} \omega_j\left\{T_{j,i_j}(x): i_j\in \varphi_j(x)\right\} 
 =\left\{\sum_{j\in J} \omega_jT_{j,i_j}(x): (i_1,\dots,i_m)\in \varphi(x)\right\},
\end{equation*} 
where $\varphi\colon X\setto I_1\times\dots\times I_m$ is defined as $\varphi(x) :=\varphi_1(x)\times\dots\times\varphi_m(x)$. Outer semicontinuity of $\varphi$ follows from Proposition~\ref{prop:osc_product}. Since $T_{j,i_j}$ is $\alpha_j$-averaged nonexpansive (resp.\ nonexpansive) for each $i_j\in I_j$ and $j\in J$, Proposition~\ref{prop:AN closedness}\ref{it:AN convex combination} implies that $\sum_{j\in J} \omega_j T_{j,i_j}$ is $\alpha$-averaged nonexpansive, for all $(i_1,\dots,i_m)\in I_1\times\dots\times I_m$, with $\alpha$ given by \eqref{eq:convex combination alpha} (resp.\ union nonexpansive) which completes the proof.

\ref{it:UAN composition}:~Using the definition of operator composition, we deduce that
\begin{equation*}
(T_m\circ \dots \circ T_1)(x) =\left\{(T_{m,i_m}\circ\dots\circ T_{1,i_1})(x): (i_1,\dots,i_m)\in \varphi(x)\right\}
\end{equation*}   
where $\varphi\colon X\setto I_1\times\dots\times I_m$ is given by
\begin{equation*}
\varphi(x) :=\{(i_1,\dots,i_m): i_1\in \varphi_1(x),\, i_2\in (\varphi_2\circ T_{1,i_1})(x),\, \dots, i_m\in (\varphi_m\circ T_{m-1,i_{m-1}}\circ\dots\circ T_{1,i_1})(x)\}.
\end{equation*}
To show that $\varphi$ is osc, consider sequences $x_n\to x$ and $i_n =(i_{1,n},\dots,i_{m,n})\to i =(i_1,\dots,i_m)$ such that $i_n\in\varphi(x_n)$. Since $I_1\times\dots\times I_m$ is finite, there exists an $n_0\in\mathbb{N}$ such that $i_n =i$ for all $n\geq n_0$. Then, for all $n\geq n_0$, $i =(i_1,\dots,i_m)\in \varphi(x_n)$, that is,
\begin{equation*}
i_{j+1}\in\begin{cases}
\phantom{(}\varphi_1(x_n) &\text{if~} j =0, \\
(\varphi_{j+1}\circ T_{j,i_j}\circ \dots \circ T_{1,i_1})(x_n) & \text{if~} j\in J\setminus\{m\}.
\end{cases}
\end{equation*}
Then, osc of $\varphi$ follows by combining Propositions~\ref{prop:osc_product}~\&~\ref{prop:osc_composition}, noting that $T_{j,i_j}$ is continuous, for each $j\in J\setminus\{m\}$.
Finally, since $T_{j,i_j}$ is $\alpha_j$-averaged nonexpansive (resp. nonexpansive) for each $i_j\in I_j$ and $j\in J$, Proposition~\ref{prop:AN closedness}\ref{it:AN compositions alpha} implies that $(T_{m,i_m}\circ\dots\circ T_{1,i_1})$ is $\alpha$-averaged nonexpansive with $\alpha$ given by \eqref{eq:composition alpha} (resp. nonexpansive) which completes the proof.
\end{proof}

For set-valued operators such as those introduced in Definition~\ref{def:union averaged}, we distinguish two different notions for fixed points which are both the same in the single-valued case. The \emph{fixed point set} of $T$ is denoted by $\Fix T :=\{x:x\in T(x)\}$, and the \emph{strong fixed point set} of $T$ is given by $\StrFix T :=\{x:T(x) =\{x\}\}$.

In the following proposition, we take a closer look the structure of a class of set-valued operators which includes union averaged nonexpansive operators as a special case. Given a set-valued operator $T\colon X\setto X$, its \emph{single-valued set}, denoted by 
\begin{equation*}
\Sing T :=\{x\in X: T(x)\text{~is a singleton}\},
\end{equation*}
is the set of points at which $T$ is single-valued.
\begin{proposition}[Active selectors, fixed points, single-valuedness]\label{prop:fixed points}
Let $T\colon X\setto X$ be a set-valued operator given by
\begin{equation*}
\forall x\in X,\quad T(x) =\{T_i(x): i\in \varphi(x)\}
\end{equation*}
where $I$ is a finite index set, $\varphi\colon X\setto I$ is an osc operator with nonempty values, and $\{T_i\}_{i\in I}$ is a collection of single-valued operators on $X$. The following assertions hold.
\begin{enumerate}[label =(\alph*)]
\item\label{it:fixed points_X} For each $i\in I$, the set $\varphi^{-1}(i)$ is closed. The space $X$ can be represented as
\begin{equation}
\label{eq:X is a union}
X =\bigcup_{i\in I}\varphi^{-1}(i).
\end{equation}    
\item\label{it:fixed points_Fix} $x\in\Fix T$ if and only if there exists $i\in I$ such that $x\in\varphi^{-1}(i)\cap\Fix T_i$. Consequently, we have
\begin{equation*}
\Fix T =\bigcup_{i\in I}\left[\varphi^{-1}(i)\cap\Fix T_i\right]. 
\end{equation*}
\item\label{it:fixed points_StrFix} $x\in\StrFix T$ if and only if 
\begin{equation*}
x\in \bigcap_{i\in\varphi(x)}\Fix T_i =\bigcap_{i\in\varphi(x)} \left[\varphi^{-1}(i)\cap\Fix T_i\right].
\end{equation*}
\item\label{it:fixed points_StrFixSing} The strong fixed point set is the intersection of the fixed point set and the single-valued set, that is,
\begin{equation*}
\StrFix T =\Fix T\cap \Sing T.
\end{equation*}
\item\label{it:fixed points_Sing} If in addition $\{T_i\}_{i\in I}$ is a collection of continuous operators, then $T$ is ocs. Consequently,
\begin{equation*}
\Sing T =\left\{x\in X:\Limsup_{x'\to x}T(x')\text{ is singleton}\right\}.
\end{equation*}
\end{enumerate}
\end{proposition}
\begin{proof}
\ref{it:fixed points_X}: Let $i\in I$ and consider a sequence $(x_n)_{n\in\mathbb{N}}\subseteq\varphi^{-1}(i)$ such that $x_n\to\overline{x}$. The definition of the inverse implies that $i\in\varphi(x_n)$ for all $n\in\mathbb{N}$ and, since $\varphi$ is osc, it holds that $i\in \varphi(\overline{x})$, which shows that $\varphi^{-1}(i)$ is closed and proves the first claim. The fact that \eqref{eq:X is a union} follows from the assumption that $\varphi(x)\neq \varnothing$ for all $x\in X$.

\ref{it:fixed points_Fix}: $x\in\Fix T\iff\exists i\in \varphi(x)$ such that $x\in\Fix T_i\iff\exists i\in I$ such that $x\in\varphi^{-1}(i)\cap\Fix T_i.$ 

\ref{it:fixed points_StrFix}: $x\in\StrFix T\iff x\in\Fix T_i$ for all $i\in\varphi(x)\iff x\in\varphi^{-1}(i)\cap\Fix T_i$ for all $i\in\varphi(x)$.

\ref{it:fixed points_StrFixSing}: Immediate from the respective definitions.

\ref{it:fixed points_Sing}: Fix $x\in X$. We first show that $T$ is osc. To this end, take $x_n\to x$ and $y_n\to y$ with $y_n\in T(x_n)$. Then the definition of $T$ ensures the existence of sequence $i_n\in \varphi(x_n)$ such that $y_n =T_{i_n}(x_n)\in T(x_n)$. Using the pigeonhole princple, we pass to a subsequence so that $y_{k_n} =T_i(x_{k_n})$ for fixed $i\in I$. The osc of $\varphi$ implies that $i\in\varphi(x)$ and continuity of $T_i$ gives
\begin{equation*}
y =\lim_{n\to\infty}y_{k_n} =\lim_{n\to\infty}T_i(x_{k_n}) =T_i(x)\in T(x),
\end{equation*}
which proves osc. The claimed formula for the singleton set follows since $x\in\Sing T\iff T(x)$ is singleton, and osc implies that $T(x) =\Limsup_{x'\to x}T(x')$. 
\end{proof}

\section{Convergence of fixed point algorithms}\label{s:convergence}
In this section, we prove our main result regarding local convergence of fixed point iterations based on union averaged nonexpansive operators. More precisely, given a union $\alpha$-averaged nonexpansive $T\colon X\setto X$ defined by $x\mapsto T(x) :=\{T_i(x):i\in\varphi(x)\}$, we study the behaviour of iterations of the form
\begin{equation*}
x_0\in X \quad\text{and}\quad \forall n\in\mathbb{N},\ x_{n+1}\in \left((1-\lambda_n)\Id+\lambda_n T\right)(x_n),
\end{equation*} 
where $(\lambda_n)_{n\in\mathbb{N}}\subseteq (0,1/\alpha]$ and $\Id$ denotes the identity operator. To do so, we first study the following closely related iterations given by
\begin{equation*}
x_0\in X \quad\text{and}\quad \forall n\in\mathbb{N},\ x_{n+1}\in \left((1-\lambda_n)\Id+\lambda_n T_{i_n}\right)(x_n),
\end{equation*} 
where each element in $I$ appears infinitely often in the sequence $(i_n)_{n\in\mathbb{N}}\subseteq I$.

The condition that each element in $I$ must appear infinitely often in a sequence is $(i_n)_{n\in\mathbb{N}}\subseteq I$ is generalized in the following definition.
\begin{definition}[Admissible sequences]
Let $I$ and $I^*$ be nonempty finite sets with $I^*\subseteq I$. A sequence $(i_n)_{n\in\mathbb{N}}\subseteq I$ is \emph{admissible in $I^*$} if every element of $I^*$ appears  infinitely often in $(i_n)_{n\in\mathbb{N}}$.
\end{definition}

\subsection{Krasnosel'ski\u{\i}--Mann iterations with admissible control}
We begin with the following variation of \cite[Theorem~1]{Els92}. Although the proof is straightforward, we include it for the sake of completeness.

\begin{theorem}[Krasnosel'ski\u{\i}--Mann iterations with admissible control]
\label{thm:KM iterations with admissible control}
Let $I$ be a finite index set and let $\{T_i\}_{i\in I}$ be a collection of nonexpansive operators on $X$ with a common fixed point. Define a sequence $(x_n)_{n\in\mathbb{N}}$ with starting point $x_0\in X$ according to
\begin{equation*}
\forall n\in\mathbb{N},\quad x_{n+1} :=(1-\lambda_n)x_n+\lambda_n T_{i_n}(x_n),
\end{equation*}
where $(i_n)_{n\in\mathbb{N}}$ is admissible in $I$, and $(\lambda_n)_{n\in\mathbb{N}}$ is in $(0, 1]$ with $\liminf_{n\to\infty}\lambda_n(1-\lambda_n)> 0$.
Then $(x_n)_{n\in\mathbb{N}}$ converges to a point $\overline{x}\in\cap_{i\in I} \Fix T_i$.
\end{theorem}
\begin{proof}
Since $\liminf_{n\to\infty}\lambda_n(1-\lambda_n)> 0$, there exists $n_0\in \mathbb{N}$ such that $\inf_{n\geq n_0}\lambda_n(1-\lambda_n)> 0$. By relabeling if necessary, we may assume without loss of generality that $\varepsilon :=\inf_{n\in\mathbb{N}}\lambda_n(1-\lambda_n)> 0$. 
Let $x\in \cap_{i\in I} \Fix T_i$ be arbitrary. Then, for all $n\in\mathbb{N}$, nonexpansiveness of $T_{i_n}$ yields
\begin{subequations} 
\begin{align*}
\|x_{n+1}-x\|^2& =\|(1-\lambda_n)(x_n-x)+\lambda_n(T_{i_n}(x_n)-x)\|^2 \\
& =(1-\lambda_n)\|x_n-x\|^2+\lambda_n\|T_{i_n}(x_n)-x\|^2-\lambda_n(1-\lambda_n)\|x_n-T_{i_n}(x_n)\|^2 \\
&\leq \|x_n-x\|^2-\varepsilon\|x_n-T_{i_n}(x_n)\|^2.
\end{align*}
\end{subequations}
It follows that $(\|x_n-x\|^2)_{n\in\mathbb{N}}$ is monotone nonincreasing, and hence convergent. Consequently, the sequence $(x_n)_{n\in\mathbb{N}}$ is bounded and 
\begin{equation}\label{eq:pre-asymp}
x_n-T_{i_n}(x_n)\to 0 \quad\text{as~} n\to\infty.
\end{equation}
Now, let $\overline{x}\in X$ be a cluster point of $(x_n)_{n\in\mathbb{N}}$. Then $(x_n)_{n\in\mathbb{N}}$ contains a convergent subsequence, say $(x_{k_n})_{n\in\mathbb{N}}$, with limit $\overline{x}$. 
We claim that $\overline{x}$ is an element of $\cap_{i\in I} \Fix T_i$. To this end suppose, by way of a contradiction, that $\overline{x}\not\in\cap_{i\in I} \Fix T_i$.
By passing to a further subsequence if necessary, we deduce that the sequence $(t_n)_{n\in\mathbb{N}}$ defined by
\begin{equation*}
t_n :=\min\{p\in\{k_n,\dots,k_{n+1}-1\}:\overline{x}\not\in\Fix T_{i_p}\}, 
\end{equation*}
is well defined.
Since $I$ is finite, by passing to yet another subsequence if necessary, we may assume that $i_{t_n} =\ell$ for all $n\in\mathbb{N}$ for some fixed index $\ell\in I$. 
Together with \eqref{eq:pre-asymp}, we deduce that
\begin{equation*}
x_{t_n}\to \overline{x} \text{~~~and~~~} (\Id-T_\ell)(x_{t_n})\to 0 \quad\text{as~} n\to\infty,
\end{equation*}
which, due to nonexpansiveness of $T_\ell$, implies that $\overline{x}\in\Fix T_\ell$. 
Thus a contradiction is obtained, and conclude that $\overline{x}\in\cap_{i\in I} \Fix T_i$. We then have that $(\|x_n-\overline{x}\|)_{n\in\mathbb{N}}$ is monotone nonincreasing and $\|x_{k_n}-\overline{x}\|\to 0$, and thus the conclusion follows.
\end{proof}

\begin{corollary} 
Let $I$ be a finite index set and let $\{T_i\}_{i\in I}$ be $\alpha_i$-averaged nonexpansive operators on $X$ with a common fixed point. Define a sequence $(x_n)_{n\in\mathbb{N}}$ with starting point $x_0\in X$ according to
\begin{equation*}
\forall n\in\mathbb{N},\quad x_{n+1} =(1-\lambda_n)x_n+\lambda_n T_{i_n}(x_n),
\end{equation*}
where $(i_n)_{n\in\mathbb{N}}$ is an admissible sequence in $I$, and $(\lambda_n)_{n\in\mathbb{N}}$ is a sequence satisfying $\lambda_n\in (0, 1/\alpha_{i_n}]$ for all $n\in\mathbb{N}$ and $\liminf_{n\to\infty}\lambda_n(1-\alpha_{i_n}\lambda_n)> 0$.
Then $(x_n)_{n\in\mathbb{N}}$ converges to a point $\overline{x}\in\cap_{i\in I} \Fix T_i$.
\end{corollary}
\begin{proof}
For each $i\in I$, by \cite[Proposition~4.35]{BauCom17}, $R_i :=(1-1/\alpha_i)\Id+(1/\alpha_i)T_i$ is nonexpansive. We also have that $\Fix R_i =\Fix T_i$ and that, for all $n\in\mathbb{N}$, $x_{n+1} =(1-\alpha_{i_n}\lambda_n)x_n+\alpha_{i_n}\lambda_n R_{i_n}(x_n)$. Now apply Theorem~\ref{thm:KM iterations with admissible control} to $\{R_i\}_{i\in I}$ and $(\alpha_{i_n}\lambda_n)_{n\in\mathbb{N}}$.
\end{proof}

\subsection{Convergence of union nonexpansive iterations}
Using the results of the previous subsection, we now turn our attention convergence of iterations based on union nonexpansive operators. Throughout this subsection, we fix a particular representation for the considered operator. Let $T\colon X\setto X$ denote a union nonexpansive (resp.\ union $\alpha$-averaged nonexpansive) operator which we assume to be represented as
\begin{equation}\label{eq:a particular uan op}
T(x) =\{T_i(x):i\in\varphi(x)\},
\end{equation}
where $I$ is a finite index set, $\{T_i\}_{i\in I}$ is a collection of nonexpansive (resp.\ $\alpha$-averaged nonexpansive) operators on $X$, and $\varphi\colon X\setto I$ is the osc, nonempty-valued active selector. Fixing this representation is convenient because, in general, the representation of a union nonexpansive (resp.\ union $\alpha$-averaged nonexpansive) operator need not be unique and allows us to avoid repetition.

Corresponding to the representation \eqref{eq:a particular uan op}, we define the \emph{radius of attraction} of $T$ at a point $x^*\in X$ as
\begin{equation}\label{eq:radius of attraction}
r(x^*;T) :=\sup\{\delta>0: \forall x\in \mathbb{B}(x^*;\delta),\ \varphi(x)\subseteq\varphi(x^*)\}
\end{equation}
Here we note that radius of attraction is nonzero for a union nonexpansive operator (and hence too for a union averaged nonexpansive operator). In fact, we have that
\begin{equation*}
\forall x^*\in X,\quad r(x^*;T)\in (0,+\infty]
\end{equation*}
as is shown in the following proposition.

\begin{proposition}\label{prop:prop2.1 Tam17}
Let $\varphi\colon X\setto I$ for a finite set $I$ and let $x^*\in X$. Then $\varphi$ is outer semicontinuous at $x^*$ if and only if there exists $\delta>0$ such that
\begin{equation*}
\forall x\in\mathbb{B}(x^*;\delta),\quad \varphi(x)\subseteq\varphi(x^*).
\end{equation*}
Consequently, if $T\colon X\setto X$ is a union nonexpansive operator, then $r(x^*;T)\in (0,+\infty]$. 
\end{proposition}
\begin{proof}
This follows from \cite[Proposition~1]{Tam17}.
\end{proof}

Furthermore, if the reference point $x^*$ in Proposition~\ref{prop:prop2.1 Tam17} is a strong fixed point of the underlying operator $T$, then $T$ satisfies the following quasi-nonexpansiveness properties.
\begin{proposition}[Radius of attraction at strong fixed points]\label{prop:radius}
Let $T\colon X\setto X$ be a union nonexpansive operator with $x^*\in\StrFix T$. Then $r :=r(x^*;T)\in (0,+\infty]$ and
\begin{equation}\label{eq:attraction nonexpansive}
\forall x\in \inte\mathbb{B}(x^*;r),\,\forall y\in T(x),\quad 
\|y-x^*\|\leq \|x-x^*\|.
\end{equation}
Furthermore, for any $\lambda\in (0, 1]$, we have
\begin{equation}\label{eq:attraction averaged}
\forall x\in \inte\mathbb{B}(x^*;r),\,\forall y\in (1-\lambda)x+\lambda T(x),\quad 
\|y-x^*\|^2+\frac{1-\lambda}{\lambda}\|x-y\|^2\leq \|x-x^*\|^2.
\end{equation}
\end{proposition}
\begin{proof}
Since $\varphi$ is osc and its range, $I$, is a finite set, Proposition~\ref{prop:prop2.1 Tam17} implies that $r :=r(x^*;T)\in(0,+\infty]$. Since $x^*\in\StrFix T$, Proposition~\ref{prop:fixed points}\ref{it:fixed points_StrFix} implies that $x^*\in\cap_{i\in\varphi(x*)}\Fix T_i$ and, since each $T_i$ is nonexpansive, we have
\begin{equation}
\label{eq:nonexpansive}
\forall i\in\varphi(x^*),\, \forall x\in X,\quad \|T_i(x)-x^*\|^2\leq \|x-x^*\|^2.
\end{equation}
In particular, for any $x\in \inte\mathbb{B}(x^*;r)$ and $y\in T(x)$, there exists an $i\in \varphi(x)\subseteq \varphi(x^*)$ such that $y =T_i(x)$. Consequently, \eqref{eq:attraction nonexpansive} follows from \eqref{eq:nonexpansive}.

Furthermore, for any $\lambda\in (0, 1]$, the operator $S_i :=(1-\lambda)\Id+\lambda T_i$ is $\lambda$-averaged nonexpansive by Proposition~\ref{prop:averaged} and we therefore have that
\begin{equation}
\label{eq:averaged}
\forall i\in\varphi(x^*),\, \forall x\in X,\quad \|S_i(x)-x^*\|^2+\frac{1-\lambda}{\lambda}\|x-S_i(x)\|^2\leq \|x-x^*\|^2.
\end{equation}
In particular, if $x\in\inte\mathbb{B}(x^*;r)$ and $y\in (1-\lambda)x+\lambda T(x)$, then there exists $i\in \varphi(x)\subseteq \varphi(x^*)$ such that $y =S_i(x)$. As before, \eqref{eq:attraction averaged} follows from \eqref{eq:averaged}.
\end{proof}

We are now ready to prove our main results which establish local convergence of Krasnosel'ski\u{\i}--Mann iterations based on union nonexpansive and union averaged nonexpansive operators.
\begin{theorem}[Local convergence of union nonexpansive iterations]\label{thm:union}
Let $T\colon X\setto X$ be a union nonexpansive operator with $x^*\in\StrFix T$ and let $(\lambda_n)_{n\in\mathbb{N}}$ be a sequence in $(0, 1]$ with $\liminf_{n\to\infty}\lambda_n(1-\lambda_n)> 0$. Denote $r :=r(x^*;T)\in(0,+\infty]$ and consider a sequence $(x_n)_{n\in\mathbb{N}}$ with $x_0\in \inte\mathbb{B}(x^*;r)$ satisfying
\begin{equation*}
\forall n\in\mathbb{N},\quad x_{n+1}\in (1-\lambda_n)x_n+\lambda_n T(x_n).
\end{equation*} 
Then $(x_n)_{n\in\mathbb{N}}$ converges to a point $\overline{x}\in\Fix T\cap \mathbb{B}(x^*;r)$. 
\end{theorem}
\begin{proof}
We first observe from Proposition~\ref{prop:radius} that, for all $n\in\mathbb{N}$, $x_n\in \inte\mathbb{B}(x^*;r)$ and 
\begin{equation*}
\forall n\in\mathbb{N},\quad \|x_{n+1}-x^*\|^2+\frac{1-\lambda_n}{\lambda_n}\|x_n-x_{n+1}\|^2\leq \|x_n-x^*\|^2
\end{equation*}
with convention that $\frac{1-\lambda_n}{\lambda_n} =0$ if $\lambda_n =0$.
By the definition of $T$, there is a sequence of indices $(i_n)_{n\in\mathbb{N}}\subseteq I$ such that
\begin{equation*}
\forall n\in\mathbb{N},\quad x_{n+1} =(1-\lambda_n)x_n+\lambda_n T_{i_n}(x_n)\quad\text{and}\quad i_n\in \varphi(x_n)\subseteq\varphi(x^*),
\end{equation*}
where the last inclusion is a consequence of Proposition~\ref{prop:prop2.1 Tam17} and the fact that $(x_n)_{n\in\mathbb{N}}\subseteq\inte\mathbb{B}(x^*;r)$.

Let $I^*$ denote the set of admissible indices in the sequence $(i_n)_{n\in\mathbb{N}}$. Then $I^*\subseteq \varphi(x^*)$ which,  together with Proposition~\ref{prop:fixed points}\ref{it:fixed points_StrFix} applied to $x^*\in \StrFix T$, yields
\begin{equation*}
x^*\in \bigcap_{i\in\varphi(x*)}\Fix T_i\subseteq \bigcap_{i\in I^*}\Fix T_i.
\end{equation*}
That is, $\{T_i\}_{i\in I^*}$ is a collection of nonexpansive operators with a common fixed point. By applying Theorem~\ref{thm:KM iterations with admissible control}, $x_n\to\overline{x}\in\cap_{i\in I^*}\Fix T_i$. Since, for all $n\in\mathbb{N}$, $x_n\in \mathbb{B}(x^*;r)$, it also holds that $\overline{x}\in \mathbb{B}(x^*;r)$. Finally, for any $i\in I^*$, there exists a subsequence $i_{k_n}\to i$. Since $x_{k_n}\to\overline{x}$ and $\varphi$ is osc with $i_{k_n}\in\varphi(x_{n_k})$, we deduce that $i\in\varphi(\overline{x})$. By Proposition~\ref{prop:fixed points}\ref{it:fixed points_Fix},  it follows that $\overline{x}\in\Fix T$ as was claimed.
\end{proof}

\begin{corollary}[Local convergence of union averaged nonexpansive iterations]\label{cor:union averaged}
Let $T\colon X\setto X$ be union $\alpha$-averaged nonexpansive operator with $x^*\in\StrFix T$ and let $(\lambda_n)_{n\in\mathbb{N}}$ be a sequence in $(0, 1/\alpha]$ with $\liminf_{n\to\infty} \lambda_n(1/\alpha-\lambda_n)> 0$. Denote $r :=r(x^*;T)\in(0,+\infty]$ and consider $(x_n)_{n\in\mathbb{N}}$ with $x_0\in \inte\mathbb{B}(x^*;r)$ satisfying
\begin{equation*}
\forall n\in\mathbb{N},\quad x_{n+1}\in (1-\lambda_n)x_n+\lambda_n T(x_n).
\end{equation*} 
Then $(x_n)_{n\in\mathbb{N}}$ converges to a point $\overline{x}\in\Fix T\cap \mathbb{B}(x^*;r)$. 
\end{corollary}
\begin{proof}
It follows from Proposition~\ref{prop:averaged} that $R :=(1-1/\alpha)\Id+(1/\alpha)T$ is union nonexpansive. Under the fixed representation of $T$ in \eqref{eq:a particular uan op}, we have $r(x^*;R) =r(x^*;T)$. Moreover, $\StrFix R =\StrFix T$, $\Fix R =\Fix T$, and, for all $n\in\mathbb{N}$, $x_{n+1} =(1-\alpha\lambda_n)x_n+\alpha\lambda_n R(x)$. The result now follows from Theorem~\ref{thm:union} applied to $R$ and $(\alpha\lambda_n)_{n\in\mathbb{N}}$. 
\end{proof}

We conclude this section with the following results concerning global convergence.

\begin{corollary}[Global convergence of union averaged nonexpansive iterations]
Let $T\colon X\setto X$ be union $\alpha$-averaged nonexpansive operator and suppose that there exists $x^*\in\StrFix T$ with $\varphi(x^*) =I$. Then 
\begin{equation}\label{eq:global attraction}
\forall x\in X,\, \forall y\in T(x),\quad \|y-x^*\|^2+\frac{1-\alpha}{\alpha}\|x-y\|^2\leq \|x-x^*\|^2.
\end{equation}
Further let $(\lambda_n)_{n\in\mathbb{N}}$ be a sequence in $(0, 1/\alpha]$ with $\liminf_{n\to\infty} \lambda_n(1/\alpha-\lambda_n)> 0$.  For any $x_0\in X$, consider a sequence $(x_n)_{n\in\mathbb{N}}$ satisfying
\begin{equation*}
\forall n\in\mathbb{N},\quad x_{n+1}\in (1-\lambda_n)x_n+\lambda_n T(x_n).
\end{equation*} 
Then $(x_n)_{n\in\mathbb{N}}$ converges to a point $\overline{x}\in\Fix T$.
\end{corollary}	
\begin{proof}
Since $\varphi(x^*) =I$, the radius of attraction \eqref{eq:radius of attraction} at $x^*$ is $r(x^*;T) =+\infty$. Equation~\eqref{eq:global attraction} now follows from Proposition~\ref{prop:radius} and convergence of $(x_n)_{n\in\mathbb{N}}$ from Corollary~\ref{cor:union averaged}.
\end{proof}

\subsection{Convergence of iterations based on compositions}
In this subsection, we look at the finer behaviour of the iterates of the compositions of union nonexpansive operators with a common fixed point. As in the previous section, it is convenient to fix representations of said operators. To this end, for each operator $T_j$ in the the finite collection of union nonexpansive operators $\{T_j\}_{j\in J}$, we fix the representation
\begin{equation*}
T_j(x) =\{T_{j,i}(x):i\in\varphi_j(x)\},
\end{equation*}
where $I_j$ is a finite index set, $\{T_{j,i}\}_{i\in {I_j}}$ is a collection of nonexpansive operators on $X$, and $\varphi_j\colon X\setto I_j$ is the osc, nonempty-valued active selector.

\begin{proposition}[Common fixed points]\label{prop:common fixed points}
Let $J :=\{1,\dots,m\}$ and let $\{T_j\}_{j\in J}$ be a collection of union averaged nonexpansive with the exception of at most one operator which is union nonexpansive. Then for each $x^*\in \cap_{i\in J}\StrFix T_j$, there exists a $\delta >0$ such that
\begin{align*}
\StrFix(T_m\circ\dots\circ T_1)\cap \mathbb{B}(x^*;\delta) & =\bigcap_{j\in J} \StrFix T_j\cap \mathbb{B}(x^*;\delta), \\
\Fix(T_m\circ\dots\circ T_1)\cap \mathbb{B}(x^*;\delta) & =\bigcap_{j\in J} \Fix T_j\cap \mathbb{B}(x^*;\delta). 
\end{align*}
\end{proposition}
\begin{proof}
We first observe that $\cap_{j\in J} \StrFix T_j\subseteq \StrFix(T_m\circ\dots\circ T_1)$ and $\cap_{j\in J} \Fix T_j\subseteq \Fix(T_m\circ\dots\circ T_1)$. Thus to prove the claimed result, we need only establish that for each $m$ there exists a $\delta >0$ such that  
\begin{subequations}
\label{eq:common Fix}
\begin{align}
\StrFix(T_m\circ\dots\circ T_1)\cap \mathbb{B}(x^*;\delta) & \subseteq \bigcap_{j\in J} \StrFix T_j, \\
\Fix(T_m\circ\dots\circ T_1)\cap \mathbb{B}(x^*;\delta) & \subseteq \bigcap_{j\in J} \Fix T_j.
\end{align}
\end{subequations}
To do so, we use induction on $m$. First, it is clear that \eqref{eq:common Fix} holds for $m=1$, so there is nothing to do. Suppose instead that \eqref{eq:common Fix} $m\geq 2$ and the result holds for $1,2,\dots,m-1$. Combining the assumptions on $\{T_i\}_{i\in I}$ with Proposition~\ref{prop:closedness}\ref{it:UAN composition}, we deduce the existence of an index $k\in J\setminus\{m\}$ such that both $S_2 :=T_m\circ\dots\circ T_{k+1}$ and $S_1 :=T_k\circ\dots\circ T_1$ are union nonexpansive, and at least one of $S_1$ or $S_2$ is union averaged nonexpansive. Then 
\begin{equation*}
x^*\in\bigcap_{i\in J}\StrFix T_j = \left(\bigcap_{i=1}^k \StrFix T_j\right)\cap \left(\bigcap_{j =k+1}^m \StrFix T_j\right)\subseteq \StrFix S_1\cap \StrFix S_2.
\end{equation*}
By the induction hypothesis, there exists a $\delta >0$ such that
\begin{align*}
\StrFix(S_1)\cap \mathbb{B}(x^*;\delta) &\subseteq \bigcap_{j =1}^k \StrFix T_j,\  &\StrFix(S_2)\cap \mathbb{B}(x^*;\delta) &\subseteq \bigcap_{j =k+1}^m \StrFix T_j, \\
\Fix(S_1)\cap \mathbb{B}(x^*;\delta) &\subseteq \bigcap_{j =1}^k \Fix T_j,\  &\Fix(S_2)\cap \mathbb{B}(x^*;\delta) &\subseteq \bigcap_{j =k+1}^m \Fix T_j.
\end{align*}
By shrinking $\delta$ if necessary, we may and do assume that $0< \delta <\min\{r(x^*;S_1), r(x^*;S_2)\}$. Now, let $x\in \StrFix (S_2\circ S_1)\cap \mathbb{B}(x^*;\delta)$. Then 
$\{x\} =S_2(y)$ for all $y\in S_1(x)$. Since $x^*\in \StrFix S_1\cap\StrFix S_2$, and either $S_1$ or $S_2$ is union $\alpha$-averaged nonexpansive, Proposition~\ref{prop:radius} applied to $S_1$ and then $S_2$ yields
\begin{equation*}
\forall y\in S_1(x),\quad \|x-x^*\|^2 \geq \|x-x^*\|^2 +\frac{1-\alpha}{\alpha}\|x-y\|^2.
\end{equation*}
Hence we have that $\|x-y\|=0$ for all $y\in S_1(x)$ or, equivalently, that $S_1(x)=\{x\}$. It then follows that $x\in \StrFix S_1\cap\StrFix S_2\cap \mathbb{B}(x^*;\delta)$ which proves the equality for the strong fixed point set. The proof for the fixed point set is performed similarly.
\end{proof}

In order to prove the main result of this subsection, we require the following technical lemma.
\begin{lemma}\label{lem:convergence of unions}
Let $\{T_j\}_{j\in J}$ be a finite collection of union nonexpansive operators on $X$ with $x^*\in \cap_{j\in J} \StrFix T_j$ and let $(\lambda_n)_{n\in\mathbb{N}}$ be a sequence in $]0, 1]$ with $\liminf_{n\to\infty} \lambda_n(1-\lambda_n)> 0$. Denote
\begin{equation*}
r :=\min_{j\in J}r(x^*;T_j)\in(0,+\infty].
\end{equation*}
and consider a sequence $(x_n)_{n\in\mathbb{N}}$ with $x_0\in\inte\mathbb{B}(x^*;r)$ satisfying
\begin{equation}
\label{eq:x+}
\forall n\in\mathbb{N},\quad x_{n+1}\in (1-\lambda_n)x_n+\lambda_n T_{j_n}(x_n),
\end{equation}
where $(j_n)_{n\in\mathbb{N}}$ is admissible in $J$. 
Then the sequence $(x_n)_{n\in\mathbb{N}}$ converges and its limit is contained in $\left(\cup_{j\in J}\Fix T_j\right)\cap \mathbb{B}(x^*;r)$.
\end{lemma}
\begin{proof}
Set $T :=\cup_{j\in J}T_j$. By Proposition~\ref{prop:closedness}\ref{it:UAN union}, $T$ is a union nonexpansive operator with active selector
\begin{equation*}
\varphi(x) :=\{(j,i): j\in J, i\in \varphi_j(x)\}.
\end{equation*}
For any $x\in X$, we have $\varphi(x)\subseteq\varphi(x^*)$ if and only if $\varphi_j(x)\subseteq\varphi_j(x^*)$ for all $j\in J$. It thus follows that
$$  r(x^*; T) =\min_{j\in J}r(x^*;T_j)\in (0,+\infty], $$
where positivity of the right-hand-side follows from Proposition~\ref{prop:prop2.1 Tam17} and the finiteness of $J$. A direct calculation then shows that
\begin{equation*}
\StrFix T =\bigcap_{j\in J}\StrFix T_j
\quad\text{and}\quad \Fix T =\bigcup_{j\in J}\Fix T_j.
\end{equation*}
We also note that every sequence generated by \eqref{eq:x+} is actually a sequence with the same starting point generated by 
\begin{equation*}
\forall n\in\mathbb{N},\quad x_{n+1}\in (1-\lambda_n)x_n+\lambda_n T(x_n).
\end{equation*}
The result thus follows from Theorem~\ref{thm:union} applied to $T$.
\end{proof}

In the following corollary, $n\bmod{m}\in\{0,\dots,m-1\}$ denotes the remainder when $n$ is divided by $m$.
\begin{corollary}[Local convergence of compositions]\label{cor:cyclic algo}
Let $J :=\{1,\dots,m\}$ and let $\{T_j\}_{j\in J}$ be a collection of union averaged nonexpansive operators on $X$ with $x^*\in \StrFix(T_m\circ \dots \circ T_1)$. Denote $r\in \left(0,r(x^*; T_m\circ \dots \circ T_1)\right]$ and consider a sequence $(x_n)_{n\in\mathbb{N}}$ with $x_0\in \inte\mathbb{B}(x^*;r)$ satisfying
\begin{equation*}
\forall n\in\mathbb{N},\quad x_{n+1}\in T_{i_n}(x_n), \text{~~where~~}i_n =\left(n\bmod{m}\right) +1.
\end{equation*}
Then $(x_{mn})_{n\in\mathbb{N}}$ converges to a point $\overline{x}\in \Fix(T_m\circ \dots \circ T_1)\cap \mathbb{B}(x^*;r)$. Furthermore, if $x^*\in \cap_{j\in J} \StrFix T_j$ and $r$ is sufficiently small, then the entire sequence $(x_{n})_{n\in\mathbb{N}}$ converges to $\overline{x}$ and $\overline{x}\in\cap_{j\in J}\Fix T_j\cap \mathbb{B}(x^*;r)$.
\end{corollary}
\begin{proof}
Since $T :=T_m\circ \dots \circ T_1$ is union averaged nonexpansive by Proposition~\ref{prop:closedness}\ref{it:UAN composition}, the claim regarding convergence of $(x_{mn})_{n\in\mathbb{N}}$ to a point $\overline{x}\in \Fix(T_m\circ \dots \circ T_1)$ follows by applying Corollary~\ref{cor:union averaged} to $T$ with all $\lambda_n =1$. To prove the second claim, first note that the entire sequence $(x_n)_{n\in\mathbb{N}}$ is convergent by Lemma~\ref{lem:convergence of unions}. Since one of its subsequence, $(x_{mn})_{n\in\mathbb{N}}$, converges to $\overline{x}$, it must be that $x_n\to\overline{x}$. Now, Proposition~\ref{prop:common fixed points} completes the proof.
\end{proof}

\section{Min-convex functions}\label{s:min-convex}
In this section we study the following class of functions whose \emph{proximity operators} will be shown to belong to the class of union averaged nonexpansive operators. 
\begin{definition}[Min-convexity]
We say a function $f\colon X\to (-\infty,+\infty]$ is \emph{min-convex} if it can be expressed in the form
\begin{equation*}
\forall x\in X,\quad f(x) :=\min_{i\in I}f_i(x),
\end{equation*}
where $I$ is a finite index set and the functions $f_i\colon X\to (-\infty,+\infty]$ are proper, lsc and convex.
\end{definition}
In general, a min-convex function need not be convex. In fact, sufficient conditions for a min-convex function to be convex were studied in \cite{BauLucPha16} (see also \cite[Proposition~5]{EbeRosSan17}). As a concrete example of a min-convex function, we revisit Example~\ref{ex:sparsity}.
\begin{example}[Sparsity projectors (revisited)]
Let $X =\mathbb{R}^n$ and $s\in\{0,1,\dots,n-1\}$. Recall that the sparsity constraint from Example~\ref{ex:sparsity} which can be expresses as the union of subspaces. \emph{i.e.,}
\begin{equation*}
C :=\{x\in\mathbb{R}^n:\|x\|_0\leq s\} =\bigcup_{I\in\mathcal{I}}C_I\text{ where }C_I :=\{x\in\mathbb{R}^n:x_i\neq 0\text{ only if }i\in I\}
\end{equation*}
Due to this representation, we see that the indicator function to $C$, $\iota_C$, can be expressed as 
$\iota_C =\min_{I\in\mathcal{I}}\iota_{C_I}$. As the indicator function to a (closed) subspace, is a proper lsc convex function, we see that $\iota_C$ is min-convex.
\end{example}

In the subsequent section, we shall study \emph{proximal algorithms} for min-convex functions. These algorithms are based upon the following two objects.
\begin{definition}[Moreau envelopes and proximity operators]
Let $f\colon X\to (-\infty,+\infty]$ be a proper function and let $\gamma>0$ be a positive parameter. The \emph{Moreau envelope} of $f$ denoted $^\gamma\! f\colon X\to (-\infty,+\infty]$ is the function
\begin{equation*}
^\gamma\! f(x) :=\inf_{y\in X}\left(f(y)+\frac{1}{2\gamma}\|x-y\|^2\right)
\end{equation*}
and the \emph{proximity operator} of $f$ denoted $\prox_{\gamma f}\colon X\setto X$ is given by
\begin{equation*}
\prox_{\gamma f}(x) =\left\{y\in X:f(y)+\frac{1}{2\gamma}\|x-y\|^2 ={}^\gamma\! f(x)\right\}.
\end{equation*}
\end{definition}
It is well known that when $f$ is proper, lsc and convex, its proximity operator is single-valued and firmly nonexpansive (\emph{i.e.,} $1/2$-averaged nonexpansive) (see, for instance, \cite[Proposition~12.28]{BauCom17}).

Recall that the \emph{proximal subdifferential} of $f\colon X\to (-\infty,+\infty]$ at $x\in X$ is given by
\begin{equation*}
\partial_p f(x) :=\left\{ x^*\in X: \exists\gamma, \delta> 0,\ \forall y\in \mathbb{B}(x;\delta),\ \langle x^*, y-x \rangle\leq f(y)-f(x)+\frac{1}{2\gamma}\|y-x\|^2 \right\}
\end{equation*}
and that $0\in \partial_p f(x)$ whenever $x$ is a local minimum of $f$; see, e.g., \cite[Equations~(0.1) and (1.4)]{MorNam07}.

In the following two propositions, we investigate various properties under assumptions which are satisfied by min-convex functions.

\begin{proposition}[Properties of proper functions]
\label{prop:basic}
Let $f\colon X\to (-\infty,+\infty]$ be proper and let $\gamma>0$. The following assertions hold.
\begin{enumerate}[label =(\alph*)]
\item\label{it:dom} $\dom {}^\gamma\! f =X$.
\item\label{it:inequalities} $\forall x\in X$,\ $\inf f(X) =\inf {}^\gamma\! f(X)\leq {}^\gamma\! f(x)\leq f(x)$. 
\item\label{it:necessary} Let $x, p\in X$. Then $p\in \prox_{\gamma f}(x)$ if and only if
\begin{equation*}
\forall y\in X,\quad \left\langle \frac{1}{\gamma}(x-p), y-p \right\rangle\leq f(y)-f(p)+\frac{1}{2\gamma}\|y-p\|^2.
\end{equation*}
In particular, if $p\in \prox_{\gamma f}(x)$, then $\frac{1}{\gamma}(x-p)\in \partial_p f(p)$.
\item\label{it:Fix} $\Fix\prox_{\gamma f} =\{x\in X: {}^\gamma\! f(x) =f(x)\}$.
\item\label{it:inclusions} The following inclusions hold.
\begin{equation*}
\argmin f\subseteq \argmin {}^\gamma\! f\cap \dom\prox_{\gamma f}\subseteq \StrFix\prox_{\gamma f}\subseteq \Fix\prox_{\gamma f}\subseteq \{x\in X: 0\in \partial_p f(x)\}.
\end{equation*}
Moreover, when $f$ is convex, all inclusions are satisfied with equality.
\end{enumerate}
\end{proposition}
\begin{proof}
\ref{it:dom} \& \ref{it:inequalities}: See, for instance, \cite[Proposition~12.9(i)--(iii)]{BauCom17}.

\ref{it:necessary}: Using definition of the $\prox_{\gamma f}$, we deduce that
\begin{subequations}
\begin{align*}
p\in \prox_{\gamma f}(x) &\iff\ \forall y\in X,\quad f(p)+\frac{1}{2\gamma}\|x-p\|^2\leq f(y)+\frac{1}{2\gamma}\|x-y\|^2 \\
&\iff\ \forall y\in X,\quad \left\langle \frac{1}{\gamma}(x-p), y-p \right\rangle\leq f(y)-f(p)+\frac{1}{2\gamma}\|y-p\|^2 \\
&\,\implies\ \frac{1}{\gamma}(x-p)\in \partial_p f(p).
\end{align*}
\end{subequations}

\ref{it:Fix}: $x\in \Fix\prox_{\gamma f} \iff x\in \prox_{\gamma f}(x)\iff {}^\gamma\! f(x) =f(x)+\frac{1}{2\gamma}\|x-x\|^2\iff {}^\gamma\! f(x) =f(x)$.

\ref{it:inclusions}: To show the first inclusion, let $x\in \argmin f$. By applying \ref{it:inequalities}, we deduce that $f(x) =\inf f(X) =\inf {}^\gamma\! f(X)\leq {}^\gamma\! f(x)\leq f(x)$, and so ${}^\gamma\! f(x) =\inf {}^\gamma\! f(X) =f(x) =f(x) +\frac{1}{2\gamma}\|x-x\|^2$, which yields $x\in \argmin {}^\gamma\! f\cap \dom\prox_{\gamma f}$. We therefore obtain that $\argmin f\subseteq \argmin {}^\gamma\! f\cap \dom\prox_{\gamma f}$.

To prove the second inclusion, let $x\in \argmin {}^\gamma\! f\cap \dom\prox_{\gamma f}$. Then there exists $p\in \prox_{\gamma f}(x)$ and using \ref{it:inequalities} we deduce that
\begin{equation*}
\inf {}^\gamma\! f(X)\leq f(p)\leq {}^\gamma\! f(x) =\min {}^\gamma\! f(X),
\end{equation*} 
which implies that $f(p) ={}^\gamma\! f(x) =f(p)+\frac{1}{2\gamma}\|x-p\|^2$,  hence we conclude that $p =x$. In other words, $\prox_{\gamma f}(x) =\{x\}$, or equivalently, $x\in \StrFix\prox_{\gamma f}$ and thus  $\argmin {}^\gamma\! f\cap \dom\prox_{\gamma f}\subseteq \StrFix\prox_{\gamma f}$.

The inclusion $\StrFix\prox_{\gamma f}\subseteq \Fix\prox_{\gamma f}$ follows immediately from the definition. To prove the final inclusion, let $x\in \Fix\prox_{\gamma f}$. Then $x\in \prox_{\gamma f}(x)$ and, by \ref{it:necessary}, $0\in \partial_p f(x)$ as was claimed. In the case in which $f$ is convex, by \cite[Proposition~7.26]{Cla13} we have $\partial_p f =\partial f$, and by \cite[Theorem~16.3]{BauCom17} we have \begin{equation*}
\{x\in X: 0\in \partial_p f(x)\} =\{x\in X: 0\in \partial f(x)\} =\argmin f,
\end{equation*}
from which the claimed equalities follow.
\end{proof}

\begin{proposition}[Properties of min-convex functions]\label{prop:proximity of min}
Let $I$ be a finite index set, let $f =\min_{i\in I} f_i$ with $f_i\colon X\to (-\infty,+\infty]$ proper, and let $\gamma> 0$. The following assertions hold.
\begin{enumerate}[label =(\alph*)]
\item\label{it:env f} $\forall x\in X,\ {}^\gamma\! f(x) =\min_{i\in I}{}^\gamma\! f_i(x)$.
\item\label{it:prox f} $\forall x\in X,\ \prox_{\gamma f}(x) =\{\prox_{\gamma f_i}(x): i\in I,\, {}^\gamma\! \!f(x) ={}^\gamma\! f_i(x)\}$.

\item\label{it:Fix prox f} $\Fix\prox_{\gamma f}\subseteq \{x\in X: 0\in \partial_p f_i(x) \text{~whenever~} f(x) =f_i(x)\}$.
\end{enumerate}
Further suppose that $f$ is min-convex (\emph{i.e.,} $f_i$ is lsc and convex, for each $i\in I$). Then	
\begin{enumerate}[label =(\alph*), resume]	
\item\label{it:local min} $\{x\in X: x \text{~is a local minimum of~} f\} =\{x\in X: x\in \argmin f_i \text{~whenever~} f(x) =f_i(x)\}$.
\item\label{it:Fix prox f_cvx} $\Fix\prox_{\gamma f}\subseteq \{x\in X: x\in \Fix\prox_{\gamma f_i} \text{~whenever~} f(x) =f_i(x)\}$. Consequently, every fixed point of $\prox_{\gamma f}$ is a local minimum of $f$.
\item\label{it:prox f_cvx} $\prox_{\gamma f}$ is union $1/2$-averaged nonexpansive. In particular, $\prox_{\gamma f}$ can be expressed as
$$\prox_{\gamma f}(x) =\{\prox_{\gamma f_i}(x): i\in \varphi(x)\}$$
with active selector $\varphi\colon X\setto I$ given by $\varphi(x) =\left\{i\in I:{}^\gamma\! f(x) ={}^\gamma\! f_i(x)\right\}$.
\end{enumerate}
\end{proposition}
\begin{proof}
\ref{it:env f}: For all $x\in X$, we have that
\begin{subequations}
\begin{align*}
^\gamma\! f(x)& =\inf_{y\in X} \Big(\min_{i\in I} f_i(y)+\frac{1}{2\gamma}\|x-y\|^2\Big) \\
& =\inf_{y\in X} \min_{i\in I} \Big(f_i(y)+\frac{1}{2\gamma}\|x-y\|^2\Big) \\
& =\min_{i\in I} \inf_{y\in X} \Big(f_i(y)+\frac{1}{2\gamma}\|x-y\|^2\Big) 
 =\min_{i\in I} {^\gamma\! f}_i(x),
\end{align*}
\end{subequations}
where interchanging infimum and minimum is valid due to the finiteness of $I$. 

\ref{it:prox f}: Suppose $p\in\prox_{\gamma f}(x)$. Since $I$ is finite, there exists an index $i\in I$ such that $f(p) =f_i(p)$. Consequently, we have 
\begin{equation*}
{}^\gamma\! f(x) =f(p)+\frac{1}{2\gamma}\|x-p\|^2 =f_i(p)+\frac{1}{2\gamma}\|x-p\|^2\geq {}^\gamma\! f_i(x).
\end{equation*}
Together with \ref{it:env f}, this implies that ${}^\gamma\! f(x) ={}^\gamma\! f_i(x)$ and that $p\in\prox_{\gamma f_i}(x)$.

To prove the reverse inclusion, suppose  $p\in\prox_{\gamma f_i}(x)$ for some $i\in I$ such that ${}^\gamma\! f(x) ={}^\gamma\! f_i(x)$. Then
\begin{equation*}
{}^\gamma\! f(x) ={}^\gamma\! f_i(x) =f_i(p)+\frac{1}{2\gamma}\|x-p\|^2\geq f(p)+\frac{1}{2\gamma}\|x-p\|^2\geq {}^\gamma\! f(x),
\end{equation*}
which implies that $p\in\prox_{\gamma f}(x)$ and thus completes the proof of \ref{it:prox f}.

\ref{it:Fix prox f}: Let $x\in\Fix\prox_{\gamma f}$. Then $0\in\partial_p f(x)$ by Proposition~\ref{prop:basic}\ref{it:inclusions}. Noting from \cite[Proposition~2.12]{MorNam07} that
\begin{equation}\label{eq:subdiff-min}
\partial_p f(x)\subseteq \bigcap_{i\in I, f(x) =f_i(x)} \partial_p f_i(x).
\end{equation}
we have $0\in \partial_p f_i(x)$ whenever $f(x) =f_i(x)$. The claim follows.

\ref{it:local min}: First note that, for every $i\in I$, the convexity of $f_i$ combined with Proposition~\ref{prop:basic}\ref{it:inclusions} implies that 
\begin{equation}\label{eq:argmin =subdiff}
\argmin f_i =\Fix\prox_{\gamma f_i} =\{x\in X: 0\in \partial_p f_i(x)\}.
\end{equation}
Now, let $x$ be a local minimum of $f$. Then $0\in \partial_p f(x)$ and, by \eqref{eq:subdiff-min}, $0\in \partial_p f_i(x)$ whenever $f(x) =f_i(x)$. Together with \eqref{eq:argmin =subdiff}, this yields $x\in \argmin f_i$ whenever $f(x) =f_i(x)$.

Conversely, consider a point $x$ such that $x\in \argmin f_i$ whenever $f(x) =f_i(x)$. Suppose, by way of a contradiction, that $x$ is not a local minimum of $f$. Then there exists a sequence $(y_n)_{n\in\mathbb{N}}$ such that $y_n\in \mathbb{B}(x; 1/n)$ and $f(y_n)< f(x)$. Set $I_0 :=\{i\in I: f(x) =f_i(x)\}$. Then
\begin{equation*}
\forall i\in I_0,\, \forall n\in\mathbb{N},\quad f(y_n)< f(x) =f_i(x)\leq f_i(y_n),
\end{equation*}
where the last inequality holds because $x\in \argmin f_i$. Therefore, for each $n\in \mathbb{N}$, there exits a $j_n\in I\smallsetminus I_0$ such that $f(y_n) =f_{j_n}(y_n)$ . As $I$ is finite, by passing to a subsequence if necessary, we can and do assume that there is $j\in I\smallsetminus I_0$ such that $f(y_n) =f_j(y_n)$ for all $n\in \mathbb{N}$. Noting that $y_n\to x$ and using lower semicontinuity of $f_j$ give
\begin{equation*}
f_j(x) \leq \liminf_{n\to\infty}  f_j(y_n) =\liminf_{n\to\infty}  f(y_n)\leq f(x) =\min_{i\in I} f_i(x),
\end{equation*}
which implies that $j\in I_0$; a contradiction. 

\ref{it:Fix prox f_cvx}: Combine \ref{it:Fix prox f}, \ref{it:local min}, and \eqref{eq:argmin =subdiff}.

\ref{it:prox f_cvx}: Using \ref{it:prox f}, we have that $\prox_{\gamma f}(x) =\{\prox_{\gamma f_i}(x): i\in \varphi(x)\}$ with $\varphi(x) =\left\{i\in I:{}^\gamma\! f(x) ={}^\gamma\! f_i(x)\right\}$ for all $x\in X$.
By \cite[Proposition~12.28]{BauCom17}, $\prox_{\gamma f_i}$ is $1/2$-averaged nonexpansive for each $i\in I$, hence only osc of the active selector $\varphi$ remains to be verified. To this end, consider sequences $(x_n,i_n)\to (x,i)$ with $i_n\in\varphi(x_n)$ for all $n\in\mathbb{N}$. Because $I$ is finite, by passing to a subsequence, we may assume that $i_n =i$ for all $n\in\mathbb{N}$. Since ${}^\gamma\! f_j$ is continuous for each $j\in J$ \cite[Proposition~12.15]{BauCom17}, we have
\begin{equation*}
{}^\gamma\! f_i(x) =\lim_{n\to\infty} {}^\gamma\! f(x_n) =\lim_{n\to\infty} \min_{j\in I}{}^\gamma\! f_j(x_n) = \min_{j\in I}\lim_{n\to\infty}{}^\gamma\! f_j(x_n) =\min_{j\in I}{}^\gamma\! f_j(x) ={}^\gamma\! f(x).
\end{equation*}
This shows that $i\in\varphi(x)$ and completes the proof. 
\end{proof}

To conclude this section, we introduce one further notion which we shall require in certain cases of our analysis. It can be viewed as a kind of constraint qualification on the representation of a min-convex function.
\begin{definition}[Outer semicontinuous representations]
Let $f\colon X\to (-\infty,+\infty]$ be a min-convex function. We say $f$ is \emph{outer semicontinuously (osc) representable} at $\overline{x}\in X$ if there exists a min-convex representation, $f =\min_{i\in I}f_i$, such that the selector $\phi\colon X\setto I$ is osc at $\overline{x}$ where
  $$ \phi(x) :=\{i\in I:f(x) =f_i(x)\}. $$
If there  exists a single representation such that $\phi$ is everywhere osc, then we say that $f$ is osc representable.
\end{definition}
Clearly every convex function is osc representable. Moreover, the following proposition shows that, in particular, the Moreau envelope of a min-convex function is also osc representable.
\begin{proposition}\label{prop:continuous rep}
Let $I$ be a finite index set, let $f_i\colon X\to\mathbb{R}$ be continuous, and set $f :=\min_{i\in I}f_i$. Then the selector $\phi(x) :=\{i\in I:f(x) =f_i(x)\}$ is osc. Consequently, the Moreau envelope of a min-convex function is always osc representable.
\end{proposition}
\begin{proof}
Let $x_n\to x$ and $i_n\to i$ with $i_n\in\phi(x_n)$ for all $n\in\mathbb{N}$. Since $I$ is finite, there exists an $n_0\in\mathbb{N}$ such that $i_n =i$ for all $n\geq n_0$. Since $\{f_i\}$ are continuous, $f$ is also continuous as the minimum of continuous functions. Consequently, 
$$ f_i(x) =\lim_{n\to\infty}f_i(x_n) = \lim_{n\to\infty}f(x_n) =f(x),$$
which shows that $i\in\phi(x)$ and establishes the osc of $\phi$. 

Now, if a function $g\colon X\to (-\infty,+\infty]$ is min-convex, there exists a finite index set, $I$, and proper, lsc convex functions $\{g_i\}_{i\in I}$ such that $g =\min_{i\in I}g_i$. For any $\gamma>0$, Proposition~\ref{prop:proximity of min}\ref{it:env f} shows that ${}^\gamma\! g =\min_{i\in I}{}^\gamma\! g_i$. Since the Moreau envelope of a proper, lsc convex function is always continuous \cite[Proposition~12.15]{BauCom17}, the claim follows.
\end{proof}

\section{Proximal algorithms for min-convex minimization}\label{s:proximal algorithms}
In this section, we use the results of the last two sections to systematically analyze proximal algorithms applied to min-convex functions. We consider four different settings: projection algorithms, the proximal point algorithm, the forward-backward method, and Douglas--Rachford splitting.

\subsection{Projection algorithms}\label{ssec:projection algorithms}
Given sets $C_1,\dots,C_m\subseteq X$ with nonempty intersection, the \emph{feasibility problem} is to
\begin{equation*}
\text{find~}x\in\bigcap_{j =1}^mC_j.
\end{equation*}
In this section, we consider the case in which each set $C_j$ is \emph{union convex} by which we mean that it can be expressed as a finite union of closed convex sets. Recall that the \emph{projector} onto a nonempty set $C$ in $X$ is the set-valued operator $P_C\colon X\setto C$ define by 
\begin{equation*}
P_C(x) :=\prox_{\iota_C}(x) =\argmin_{c\in C} \|x-c\| =\{c\in C: \|x-c\| =\dist(x, C)\},
\end{equation*}
where $\dist(x, C) :=\inf_{c\in C} \|x-c\|$ is the \emph{distance} from $x$ to $C$.

\begin{proposition}[Union convex sets]\label{prop:projection operators}
Let $A =\cup_{i\in I}A_i$ and $B =\cup_{j\in J}B_j$ where $I,J$ are finite index sets and $A_i, B_j$ are nonempty closed convex sets in $X$. The following assertions hold.
\begin{enumerate}[label =(\alph*)]
\item\label{it:proj} The projector $P_A$ is union $1/2$-averaged nonexpansive with 
\begin{equation*}
P_A(x) =\{P_{A_i}(x):i\in I,\, \dist(x,A_i) =\dist(x,A)\}
\end{equation*}
and the \emph{reflector} $R_A :=2P_A-\Id$ is union nonexpansive.
\item\label{it:DR op} The Douglas--Rachford (DR) operator given by
\begin{equation*}
T_{A,B}(x) :=\frac{\Id+R_B\circ R_A}{2}(x) =\{ x+b-a\in X: a\in P_A(x),\,b\in P_B(2a-x)\}
\end{equation*}
is union $1/2$-averaged nonexpansive.
\end{enumerate}
\end{proposition}
\begin{proof}
\ref{it:proj}: Let $f_i =\iota_{A_i}$ in Proposition~\ref{prop:proximity of min}\ref{it:prox f_cvx} and then apply Proposition~\ref{prop:averaged}.

\ref{it:DR op}: Since both $R_A$ and $R_B$ are union nonexpansive by \ref{it:proj}, Proposition~\ref{prop:closedness}\ref{it:UAN composition} implies that $R_BR_A$ is also union nonexpansive. The result then follows by applying Proposition~\ref{prop:averaged}\ref{it:prop averaged 2} to $R_BR_A$ with $\alpha =1/2$.
\end{proof}

\begin{remark}
In the setting of Proposition~\ref{prop:projection operators}, one can also deduce union averaged nonexpansiveness of relaxations of projection operators such as those consider in the so-called \emph{generalized Douglas--Rachford operator} \cite{DaoPha16,DaoPha17} 
which includes the \emph{relaxed averaged alternating reflection operator} \cite{Luk08}.
However, we shall focus on algorithms involving projectors and (ungeneralized) Douglas--Rachford operators.
\end{remark}

We now state our results regarding convergence of projection algorithms. We consider three different algorithms: the \emph{method of cyclic projections}, the \emph{cyclic Douglas--Rachford method} \cite{BorTam14,BorTam15} and the \emph{cyclically anchored Douglas--Rachford method} \cite{BauNolPha15}. The latter includes the usual two-set Douglas--Rachford method as a special case.

\begin{theorem}[Projection algorithms on union convex sets]\label{thm:proj}
Let $J :=\{1,\dots,m\}$ and let $\{C_j\}_{j\in J}$ be a finite collection of union convex sets in $X$. Given $x_0\in X$, define $x_{n+1}\in T(x_n)$ for all $n\in\mathbb{N}$ in any one of the following cases.
\begin{enumerate}[label =(\alph*)]
\item 
 (method of cyclic projections) $T =P_{C_m}\circ\dots\circ P_{C_2}\circ P_{C_1}$.
\item
 (cyclic Douglas--Rachford method) $T =T_{C_m,C_1}\circ\dots\circ T_{C_2,C_3}\circ T_{C_1,C_2}$.
\item 
 (cyclically anchored Douglas--Rachford method) $T =T_{C_1,C_m}\circ\dots\circ T_{C_1,C_3}\circ T_{C_1,C_2}$.
\end{enumerate}   
Then $\cap_{j \in J} C_j\subseteq \StrFix T$. Moreover, if $x^*\in\StrFix T$, then there exists $r> 0$ such that, whenever $x_0\in \inte\mathbb{B}(x^*;r)$, the sequence $(x_n)_{n\in\mathbb{N}}$ converges to a point $\overline{x}\in\Fix T\cap \mathbb{B}(x^*;r)$. 
\end{theorem}
\begin{proof}
It is straightforward to check that $\cap_{j \in J} C_j\subseteq \StrFix T$. To show the second claim, first combine Proposition~\ref{prop:closedness}\ref{it:UAN composition} and Proposition~\ref{prop:projection operators} to deduce that $T$ is union averaged nonexpansive. The results then follows from Corollary~\ref{cor:union averaged} with $\lambda_n =1$.
\end{proof}

\begin{theorem}[Cyclically anchored Douglas--Rachford method]\label{thm:CADR}
Let $J :=\{1,\dots,m\}$ and suppose $\{C_j\}_{j\in J}$ is a finite collection of union convex sets in $X$. Consider a sequence $(x_n)_{n\in\mathbb{N}}$ with $x_0\in X$ satisfying
\begin{equation*}
\forall n\in\mathbb{N},\quad x_{n+1}\in T_{C_1,C_{i_n}}(x_n)\text{~~where~~}i_n =\left(n\bmod{(m-1)}\right)+2.
\end{equation*}
If $x^*\in \cap_{j\in J\setminus\{1\}} \StrFix T_{C_1,C_j}$ (in particular, if $x^*\in \cap_{i\in J}C_j$), then there exists an $r> 0$ such that the sequence $(x_n)_{n\in\mathbb{N}}$ converges to a point $\overline{x}\in \cap_{j\in J\setminus\{1\}} \Fix T_{C_1,C_j}$ whenever $x_0\in \inte\mathbb{B}(x^*;r)$. Moreover, if the set $C_1$ is convex, then  $P_{C_1}(\overline{x})\in \cap_{j \in J} C_j$.
\end{theorem}
\begin{proof}
Let $x^*\in \cap_{j\in J\setminus\{1\}} \StrFix T_{C_1,C_j}$.
For every $j\in J\setminus\{1\}$, the definition of $T_{C_1,C_j}$ yields $C_1\cap C_j\subseteq \StrFix T_{C_1,C_j}$. Consequently, we have
   $$\cap_{j \in J} C_j\subseteq \cap_{j\in J\setminus\{1\}} \StrFix T_{C_1,C_j}\subseteq \StrFix ( T_{C_1,C_m}\circ\dots\circ T_{C_1,C_2}), $$   
which shows, in particular, that $x^*$ is strong fixed point of the cyclically anchored Douglas--Rachford operator. By Proposition~\ref{prop:projection operators}\ref{it:DR op}, $\{T_{C_1,C_j}\}_{j\in J\setminus\{1\}}$ is a collection of union $1/2$-averaged operators. By setting $r :=\min_{j\in J\setminus\{1\}} r(x^*;T_{C_1,C_j})> 0$ and applying Corollary~\ref{cor:cyclic algo}, we deduce that the sequence $(x_n)_{n\in\mathbb{N}}$ converges to a point $\overline{x}\in \cap_{j\in J\setminus\{1\}} \Fix T_{C_1,C_j}$ whenever $x_0\in \inte\mathbb{B}(x^*;r)$, which proves the first claim. Moreover, if $C_1$ is convex, then \cite[Equation~(23)]{BauDao17} implies that $P_{C_1}(\overline{x})\in C_1\cap C_j$ for every $j\in J\setminus\{1\}$, which completes the proof. 
\end{proof}

In particular, setting $m =2$ in Theorem~\ref{thm:CADR}, we recover the result for the usual two-set Douglas--Rachford method as a special case. This was original proven by Bauschke and Noll \cite[Theorem~1]{BauNol14}.

\begin{corollary}[Douglas--Rachford method {\cite[Theorem~1]{BauNol14}}]
Let $C_1$ and $C_2$ be union convex sets in $X$. Consider a sequence $(x_n)_{n\in\mathbb{N}}$ with $x_0\in X$ satisfying
\begin{equation*}
\forall n\in\mathbb{N},\quad x_{n+1}\in T_{C_1,C_{2}}(x_n).
\end{equation*}
If $x^*\in \StrFix T_{C_1,C_2}$ (in particular, if $x^*\in C_1\cap C_2$), then there exists an $r> 0$ such that the sequence $(x_n)_{n\in\mathbb{N}}$ converges to a point $\overline{x}\in\Fix T_{C_1,C_2}$ whenever $x_0\in \inte\mathbb{B}(x^*;r)$. Moreover, there exists a point $c\in P_{C_1}(\overline{x})$ such that $c\in C_1\cap C_2$.
\end{corollary}
\begin{proof}
By applying Theorem~\ref{thm:CADR} with $m =2$, it follows that $(x_n)_{n\in\mathbb{N}}$ converges to a point $\overline{x}\in\Fix T_{C_1,C_2}$. Using \cite[Equation~(22)]{BauDao17} implies the existence of a point $c\in P_{C_1}(\overline{x})$ such that $c\in P_{C_2}(2c-\overline{x})$. Consequently, $c\in C_1\cap C_2$ which completes the proof.
\end{proof}

In the following corollary, we deduce the corresponding convergence result for the method of cyclic projections by observing that it can be cast as a special instance of the cyclically anchored Douglas--Rachford method.
\begin{corollary}[Method of cyclic projections]\label{cor:MAP}
Let $J :=\{1,\dots,m\}$ and let $\{C_j\}_{j\in J}$ be a finite collection of union convex sets in $X$ with $x^*\in \cap_{j\in J} C_j$. Define a sequence $(x_n)_{n\in\mathbb{N}}$ with starting point $x_0\in X$ according to
\begin{equation*}
\forall n\in\mathbb{N},\quad x_{n+1}\in P_{C_{i_n}}(x_n), \text{~~where~~}i_n =\left(n\bmod{m}\right)+1.
\end{equation*}
Then there exists an $r> 0$ such that $(x_n)_{n\in\mathbb{N}}$ converges to a point $\overline{x}\in \cap_{j\in J} C_j$ whenever $x_0\in \inte\mathbb{B}(x^*;r)$. 
\end{corollary}
\begin{proof}
We first note that, for a nonempty closed set $C$, $T_{X,C} =P_C$ and $\StrFix P_C =\Fix P_C =C$. The result now follows by applying Theorem~\ref{thm:CADR} to $\{X, C_1, \dots, C_m\}$; a collection of union convex sets with the first set, $X$, being convex.
\end{proof}

\begin{remark}[Sparse affine feasibility]
Given a matrix $A$, point $b\in\range(A)$ and a sparsity bound $s$, the \emph{sparse affine feasibility} problem asks for a point $x$ such that $Ax =b$ and $\|x\|_s\leq s$. The \emph{method of alternating projections} (\emph{i.e.,} Corollary~\ref{cor:MAP} with $m =2$) applied to this problem has been studied by \cite{HesLukNeu14} who used \emph{regularity notions} to show local linear convergence of the method of alternating projections.  Whilst our Corollary~\ref{cor:MAP} does applied to deduce local convergence for this problem, it does not say anything about the rate.
\end{remark}

\begin{remark}[cyclic Douglas--Rachford method]
It is not clear if the conclusions of Theorem~\ref{thm:proj} can be improved for the cyclic Douglas--Rachford method, specially, if it can be shown projectors of the limit $\overline{x}$ can be used to produce a point in the intersection $\cap_{j\in J}\cap C_j$, as in the case in the convex setting \cite{BorTam14}. To illustrate the difficulty, consider the case when $J =\{1,2,3\}$. In this case, Theorem~\ref{thm:proj} gives that the limit $\overline{x}$ satisfies $ \overline{x} \in \Fix\left(T_{C_3,C_1}\circ T_{C_2,C_3}\circ T_{C_1,C_2} \right). $ From this it is only possible to deduce that there exists convex subsets $C_j',C_j''\subseteq C_j$, for $j\in J$, with 
  $$ \overline{x} \in \Fix\left(T_{C_3'',C_1''}\circ T_{C_2'',C_3'}\circ T_{C_1',C_2'} \right), $$
but where it is not necessarily the case that $C_j' =C_j''$. 
\end{remark}

\subsection{The proximal point algorithm} 
In this section, we consider the minimization problem
\begin{equation}\label{eq:min g}
\min_{x\in X} g(x)
\end{equation}
where $g\colon X\to (-\infty,+\infty]$ is min-convex. 

Let $\gamma>0$. Given $x_0\in X$, the \emph{proximal point algorithm} (with fixed stepsize) for \eqref{eq:min g} is given by
\begin{equation}\label{eq:PPA}
\forall n\in\mathbb{N},\quad  x_{n+1} \in T_{\rm PPA}(x_n) :=\prox_{\gamma g}(x_n).
\end{equation}
By applying our main convergence result, we are able to deduce the following result regarding convergence of the proximal point algorithm the min-convex function $g$.
\begin{theorem}[Proximal point algorithm]
Let $g\colon X\to (-\infty,+\infty]$ be a min-convex function. Suppose that $x^*\in\StrFix T_{\rm PPA}$. Denote $r :=r(x^*;T_{\rm PPA})\in(0,+\infty]$ and consider a sequence $(x_n)_{n\in\mathbb{N}}$ given by \eqref{eq:PPA} with $x_0\in\inte\mathbb{B}(x^*;r)$. Then $(x_n)_{n\in\mathbb{N}}$ converges to a local minimum of $g$.
\end{theorem}
\begin{proof}
By Proposition~\ref{prop:proximity of min}\ref{it:prox f_cvx}, the operator $T_{\rm PPA}$ is union $1/2$-averaged nonexpansive. Consequently, convergence of the sequence $(x_n)_{n\in\mathbb{N}}$ to a point in $\Fix T_{\rm PPA}$ then follows from Corollary~\ref{cor:union averaged} (with $\lambda_n =1$). 
The fact that every fixed point of $T_{\rm PPA}$ is a local minimum of $g$ follows from Proposition~\ref{prop:proximity of min}\ref{it:Fix prox f_cvx}.
\end{proof}

\subsection{Forward-backward splitting}
In this section, we consider the minimization problem
\begin{equation}\label{eq:min f+g fbs}
\min_{x\in X}\{f(x)+g(x)\}
\end{equation}
where $f\colon X\to\mathbb{R}$ is a convex function with $L$-Lipschitz continuous gradient $\nabla f$, and $g :=\min_{i\in I} g_i\colon X\to (-\infty,+\infty]$ is a min-convex function.

Given $x_0\in X$, the \emph{forward-backward algorithm} for \eqref{eq:min f+g fbs} is given by fixed point iteration
\begin{equation}\label{eq:FB sequence}
\forall n\in \mathbb{N},\quad x_{n+1}\in (1-\lambda_n)x_n+\lambda_n T_{\rm FB}(x_n)\text{~~with~~}T_{\rm FB} :=\prox_{\gamma g}(\Id-\gamma \nabla f),
\end{equation} 
where $\gamma \in (0,2/L)$ and $(\lambda_n)_{n\in\mathbb{N}}\subseteq (0,(4-\gamma L)/2]$. 

\begin{remark}
In the special case when $g_i$ are indicator functions to convex sets, the proximity operator $\prox_{\gamma g}$ reduces to a projection operators, and the corresponding algorithm is sometimes called the \emph{projected gradient algorithm}. A specific example of which was studied in \cite{Tam17} arising is sparsity constrained minimization.
\end{remark}

To begin, we study some properties of the forward-backward operator.
\begin{proposition}[Properties of $T_{\rm FB}$]\label{prop:properties of FB}
Let $f\colon X\to\mathbb{R}$ be a convex function with $L$-Lipscthiz continuous gradient, $g :=\min_{i\in I} g_i\colon X\to (-\infty,+\infty]$ be a min-convex function, and $\gamma\in(0,2/L)$. Then the following assertions hold.
\begin{enumerate}[label =(\alph*)]
\item\label{it:FBS averaged} The forward-backward splitting operator, $T_{\rm FB}$, is union $2/(4-\gamma L)$-averaged nonexpansive.

\item\label{it:FBS Fix critical} If $x\in\Fix T_{\rm FB}$, then there exists an $i\in I$ such that $0\in \left(\nabla f(x)+\partial_p g(x)\right)\cap\left(\nabla f(x)+\partial g_i(x)\right) $.

\item\label{it:FBS Fix nec}  $x\in \Fix T_{\rm FB}$ if and only if there exists an $i\in I$ such that $x\in \argmin\{f+g_i\}$ and ${}^\gamma\! g(x-\gamma\nabla f(x)) ={}^\gamma\! g_i(x-\gamma\nabla f(x))$, in which case $g(x) =g_i(x)$.

\item\label{it:FBS sFix nec} $x\in \StrFix T_{\rm FB}$ if and only if $x\in \argmin\{f+g_i\}$ for all $i\in I$ such that ${}^\gamma\! g(x-\gamma\nabla f(x)) ={}^\gamma\! g_i(x-\gamma\nabla f(x))$, in which case $g(x) =g_i(x)$.

\item\label{it:FBS sFix local min} $\StrFix T_{\rm FB} \subseteq \{x\in X:x\text{ is a local minimum of }f+g\}\cap \Fix T_{\rm FB}$.	

\item\label{it:FBS Fix local min} If $x^*\in\StrFix T_{\rm FB}$, then there exists a $\delta>0$ such that $$\Fix T_{\rm FB}\cap\mathbb{B}(x^*;\delta)\subseteq \{x\in X:x\text{ is a local minimum of }f+g\}.$$
\end{enumerate}	
\end{proposition}
\begin{proof}
\ref{it:FBS averaged}: First, by \cite[Proposition~4.39 and Theorem~18.15(i)\&(v)]{BauCom17}, $\Id-\gamma\nabla f$ is $(\gamma L)/2$-averaged nonexpansive and, by Proposition~\ref{prop:proximity of min}\ref{it:prox f_cvx}, $\prox_{\gamma g}$ is union $1/2$-averaged nonexpansive, in particular, 
\begin{equation*}
\prox_{\gamma g}(x) =\{ \prox_{\gamma g_i}(x): i\in\varphi(x)\}, 
\end{equation*}
where $\varphi(x) :=\{i\in I:{}^\gamma\! g_i(x) ={}^\gamma\! g(x)\}$. Applying Proposition~\ref{prop:closedness}\ref{it:UAN composition} it follows that $T_{\rm FB}$ is union $2/(4-\gamma L)$-averaged nonexpansive with 
\begin{equation}
T_{\rm FB}(x) =\{ \prox_{\gamma g_i}(x-\gamma\nabla f(x)): i\in \phi(x\},
\end{equation}
where $\phi :=\varphi\circ(\Id-\gamma\nabla f)$ is osc.

\ref{it:FBS Fix critical}: Let $x\in\Fix T_{\rm FB}$. Then $x\in \prox_{\gamma g}(x-\gamma\nabla f(x))$. Applying Proposition~\ref{prop:basic}\ref{it:necessary} to $\prox_{\gamma g}$ gives $\frac{1}{\gamma}(x-\gamma\nabla f(x)-x) \in \partial_p g(x)$ which implies that $0\in \nabla f(x)+\partial_p g(x)$. Further, Proposition~\ref{prop:proximity of min}\ref{it:prox f_cvx} implies that there exists an $i\in I$ such that $x= \prox_{\gamma g_i}(x-\gamma\nabla f(x))$ and hence $0\in \nabla f(x)+\partial g_i(x)$.

\ref{it:FBS Fix nec}: Let $x\in X$ and set $y :=x-\gamma \nabla f(x)$. Then $x\in \Fix T_{\rm FB}$ if and only if there exists an $i\in I$ such that ${}^\gamma\! g(y) ={}^\gamma\! g_i(y)$ and $x =\prox_{\gamma g_i}(y)$. Since $f,g_i$ are proper, lsc and convex and $f$ has full domain, the latter is equivalent to $x\in\argmin\{f+g_i\}$ \cite[Corollary~27.3(i)\&(viii)]{BauCom17}, thus establishing the claimed equivalence.

To prove the second claim, we note that $x =\prox_{\gamma g_i}(y)\in \prox_{\gamma g}(y)$ and ${}^\gamma\! g(y) ={}^\gamma\! g_i(y)$ implies
\begin{equation*}
g(x) ={}^\gamma\! g(y) -\frac{1}{2\gamma}\|y-x\|^2 
 ={}^\gamma\! g_i(y)  -\frac{1}{2\gamma}\|y-x\|^2 
 =g_i(x),
\end{equation*}
as was claimed.

\ref{it:FBS sFix nec}: Let $x\in X$ and set $y :=x-\gamma \nabla f(x)$. Then $x\in \StrFix T_{\rm FB}$ if and only if $x =\prox_{\gamma g_i}(y)$ for all $i\in I$ such that ${}^\gamma\! g(y) ={}^\gamma\! g_i(y)$. The rest is analogous to that of \ref{it:FBS Fix nec}.

\ref{it:FBS sFix local min}: We always have $\StrFix T_{\rm FB}\subseteq \Fix T_{\rm FB}$. Let $x\in \StrFix T_{\rm FB}$ and set $y :=x-\gamma \nabla f(x)$. It follows from \ref{it:FBS sFix nec} that 
\begin{equation}
\label{eq:FBS SFix}
\forall i\in I,\quad {}^\gamma\! g(y) ={}^\gamma\! g_i(y) \implies x\in \argmin\{f+g_i\}.
\end{equation}
We shall prove that $x$ is a local minimum of $f+g$. Following Proposition~\ref{prop:proximity of min}\ref{it:local min} and \eqref{eq:FBS SFix}, it suffices to show that
\begin{equation*}
\forall i\in I,\quad g(x) =g_i(x) \implies {}^\gamma\! g(y) ={}^\gamma\! g_i(y).
\end{equation*}
Indeed, let $i\in I$ be arbitrary such that $g(x) =g_i(x)$. Since $x\in\prox_{\gamma g}(y)$, using Proposition~\ref{prop:proximity of min}\ref{it:env f}, we have
\begin{equation*}
{}^\gamma\! g_i(y)\geq {}^\gamma\! g(y) =g(x) +\frac{1}{2\gamma}\|x-y\|^2
 =g_i(x) +\frac{1}{2\gamma}\|x-y\|^2 
\geq {}^\gamma\! g_i(y), 
\end{equation*}
which implies that ${}^\gamma\! g(y) ={}^\gamma\! g_i(y)$. 

\ref{it:FBS Fix local min}:~Let $x^*\in\StrFix T_{\rm FB}$. Using \ref{it:FBS sFix local min}, there exists $\delta_1>0$ such that $x^*$ is a minimum of $f+g$ on $\mathbb{B}(x^*;\delta_1)$. By Proposition~\ref{prop:prop2.1 Tam17}, there exists $\delta_2>0$ such that 
\begin{equation*}
\forall x\in \mathbb{B}(x^*;\delta_2),\quad \phi(x)\subseteq \phi(x^*).
\end{equation*}
Now take $\delta \in (0,\min\{\delta_1, \delta_2\})$ and let $\overline{x}\in \Fix T_{\rm FB}\cap \mathbb{B}(x^*;\delta)$. Then, by \ref{it:FBS Fix nec}, there exists $i\in \phi(\overline{x})$ such that $\overline{x}\in \argmin\{f+g_i\}$ and $g(\overline{x}) =g_i(\overline{x})$. As $x^*\in\StrFix T_{\rm FB}$ and $i\in \phi(\overline{x})\subseteq \phi(x^*)$, we derive from \ref{it:FBS sFix nec} that $x^*\in \argmin\{f+g_i\}$ and $g(x^*) =g_i(x^*)$. Therefore, 
\begin{equation*}
f(\overline{x})+g(\overline{x}) =f(\overline{x})+g_i(\overline{x}) =f(x^*)+g_i(x^*) =f(x^*)+g(x^*).
\end{equation*} 
Since $\overline{x}\in B(x^*;\delta)\subset\inte\mathbb{B}(x^*;\delta_1)$ with the same value as $f+g$ at $x^*$, $\overline{x}$ is also a local minimum of $f+g$.
\end{proof}

Our main results regarding convergence of the forward-backward method can now be stated as follows.
\begin{theorem}[Forward-backward splitting]
Let $f\colon X\to\mathbb{R}$ be a convex function with $L$-Lipscthiz continuous gradient, $g\colon X\to (-\infty,+\infty]$ be a min-convex function, and $\gamma\in(0,2/L)$. 
Suppose that $x^*\in\StrFix T_{\rm FB}$ and let $(\lambda_n)_{n\in\mathbb{N}}$ be a sequence in $\left(0,\frac{4-\gamma L}{2}\right]$ with $\liminf_{n\to\infty}\lambda_n\left(\frac{4-\gamma L}{2}-\lambda_n\right)>0$. Denote $r :=r(x^*;T_{\rm FB})\in(0,+\infty]$ and consider a sequence $(x_n)_{n\in\mathbb{N}}$ given by \eqref{eq:FB sequence} with $x_0\in\inte\mathbb{B}(x^*;r)$. Then $(x_n)_{n\in\mathbb{R}}$ converges to a point in $\overline{x}\in\Fix T_{\rm FB}\cap\mathbb{B}(x^*;r)$. Furthermore, for sufficiently small $r>0$, the limit point $\overline{x}$ is a local minimum of $f+g$.
\end{theorem}
\begin{proof}
Convergence of the sequence $(x_n)_{n\in\mathbb{N}}$ to a point $\overline{x}\in\Fix T_{\rm FB}\cap\mathbb{B}(x^*;r)$ follows by combining Proposition~\ref{prop:properties of FB}\ref{it:FBS averaged} and Corollary~\ref{cor:union averaged}. For sufficiently small $r>0$, the fact that $\overline{x}$ is a local minimum of $f+g$ follows from Proposition~\ref{prop:properties of FB}\ref{it:FBS Fix local min}.
\end{proof}

In the proof of the previous theorem, we used an inclusion relating fixed points and local minima in Proposition~\ref{prop:properties of FB}\ref{it:FBS Fix local min}. To conclude our study of the forward-backward method, we show that osc representability gives the reverse inclusion. 
\begin{proposition}[Local minima are fixed points]
Let $f\colon X\to\mathbb{R}$ be convex with $L$-Lipscthiz continuous gradient, let $g :=\min_{i\in I} g_i\colon X\to (-\infty,+\infty]$ be min-convex and osc representable at $\overline{x}\in X$. Suppose that $\inf(f+g)(X)>-\infty$ and that $\overline{x}$ is a local minimum of $f+g$. There exist constants $\overline{\gamma}, \delta>0$ such that, for each $x\in\mathbb{B}(\overline{x};\delta)$ and $\gamma\in(0,\overline{\gamma}]$, there exists an index $i\in I$ satisfying:
\begin{enumerate}[label =(\alph*)]	
\item $\overline{x}\in\argmin\{f+g_i\}$,
\item ${}^\gamma\! g(x-\gamma\nabla f(x)) ={}^\gamma\! g_i(x-\gamma\nabla f(x))$, and 
\item $\prox_{\gamma g_i}(x-\gamma\nabla f(x))\in\prox_{\gamma g}(x-\gamma\nabla f(x))$.		
\end{enumerate}
Consequently, $\overline{x}\in\Fix T_{\rm FB}$ for all $\gamma\in(0,\overline{\gamma}]$.
\end{proposition}
\begin{proof}
Assume that $\overline{x}$ is a local minimum of $f+g$ and consider the selector $\phi\colon X\setto I$ defined by 
\begin{equation*}
\phi(x) :=\{i\in I: g(x) =g_i(x)\}.
\end{equation*}
Then, by Proposition~\ref{prop:proximity of min}\ref{it:local min}, $\overline{x}\in \argmin\{f+g_i\}$ for all $i\in \phi(\overline{x})$. Since $\phi$ is osc at $\overline{x}$ by assumption, Proposition~\ref{prop:prop2.1 Tam17} implies that there exists $\delta >0$ such that 
\begin{equation}
\label{eq:phi(x)}
\forall x\in \mathbb{B}(\overline{x};3\delta),\quad \phi(x)\subseteq \phi(\overline{x}).
\end{equation}
Fix a constant $\overline{\gamma}\in(0,\frac{1}{2L})$ satisfying
\begin{equation}\label{eq:gamma f+g ub}
2\overline{\gamma}\left((f+g)(\overline{x})-\inf(f+g)(X)\right) < \delta^2.
\end{equation}
Let $\gamma\in(0,\overline{\gamma}]$, $x\in\mathbb{B}(\overline{x};\delta)$ and $z\in\prox_{\gamma g}(x-\gamma\nabla f(x))$.
Then 
\begin{equation*}
g(\overline{x}) +\frac{1}{2\gamma}\|x-\gamma \nabla f(x)-\overline{x}\|^2 
\geq {}^\gamma\! g(x-\gamma \nabla f(x)) 
 =g(z) +\frac{1}{2\gamma}\|x-\gamma \nabla f(x)-z\|^2.
\end{equation*} 
Since $f$ is convex with $L$-Lipschitz continuous gradient, \cite[Proposition~17.7(ii) and Theorem~18.15(iii)]{BauCom17} implies that
$$ f(z) \leq f(\overline{x}) +\langle \nabla f(x),z-\overline{x}\rangle +\frac{L}{2}\|z-x\|^2. $$
Combining the previous two equations gives
\begin{align*}
2\gamma\left( g(\overline{x})-g(z)\right)
&\geq \|x-\gamma \nabla f(x)-z\|^2-\|x-\gamma \nabla f(x)-\overline{x}\|^2 \\
&= \|x-z\|^2  -\|x-\overline{x}\|^2 +2\gamma\langle \nabla f(x),z-\overline{x}\rangle \\
&\geq \|x-z\|^2  -\|x-\overline{x}\|^2 +2\gamma\left( f(z)-f(\overline{x})-\frac{L}{2}\|z-x\|^2\right),
\end{align*}
which together with \eqref{eq:gamma f+g ub} yields
\begin{equation*}
\delta^2 >2\gamma\left((f+g)(\overline{x})-\inf(f+g)(X)\right) \geq  (1-\gamma L)\|x-z\|^2  -\|x-\overline{x}\|^2 
\geq \frac{1}{2}\|x-z\|^2 -\delta^2. 
\end{equation*}
We obtain that $\|x-z\|\leq 2\delta$ and so $\|z-\overline{x}\|\leq \|z-x\| +\|x-\overline{x}\|\leq 3\delta$. Therefore, $z\in\mathbb{B}(\overline{x};3\delta)$ and, by \eqref{eq:phi(x)}, $\phi(z)\subseteq \phi(\overline{x})$. Take any $i\in \phi(z)$. On the one hand, $i\in \phi(\overline{x})$, and so $\overline{x}\in\argmin \{f+g_i\}$. On the other hand, $g(z) =g_i(z)$, which implies that
\begin{align*}
{}^\gamma\! g(x-\gamma\nabla f(x)) 
& =g(z) +\frac{1}{2\gamma}\|x-\gamma\nabla f(x)-z\|^2 \\
& =g_i(z) +\frac{1}{2\gamma}\|x-\gamma\nabla f(x)-z\|^2 \\
&\geq {}^\gamma\! g_i(x-\gamma\nabla f(x)) \geq {}^\gamma\! g(x-\gamma\nabla f(x)).
\end{align*} 
It follows that ${}^\gamma\! g(x-\gamma\nabla f(x)) ={}^\gamma\! g_i(x-\gamma\nabla f(x))$ and also $\prox_{\gamma g_i}(x-\gamma\nabla f(x))\in\prox_{\gamma g}(x-\gamma\nabla f(x))$.In particular, ${}^\gamma\! g(\overline{x}-\gamma\nabla f(\overline{x})) ={}^\gamma\! g_i(\overline{x}-\gamma\nabla f(\overline{x}))$ and hence, by appealing to Proposition~\ref{prop:properties of FB}\ref{it:FBS Fix nec}, we deduce that $\overline{x}\in \Fix T_{\rm FB}$.
\end{proof}

\subsection{Douglas--Rachford splitting}

In this section, we consider the minimization problem
\begin{equation}\label{eq:min f+g}
\min_{x\in X}\{f(x)+g(x)\},
\end{equation} 
where $f\colon X\to (-\infty,+\infty]$ and $g\colon X\to (-\infty,+\infty]$ are min-convex functions.

Given $x_0\in X$ and sequence $(\lambda_n)_{n\in\mathbb{N}}\subseteq (0,2]$, the \emph{Douglas--Rachford splitting algorithm} for \eqref{eq:min f+g} can be described as iteration
\begin{equation}\label{eq:DR splitting}
\forall n\in\mathbb{N},\quad \left\{\begin{array}{rcl}
y_n     &\in   & \prox_{\gamma f}(x_n), \\
z_n     &\in & \prox_{\gamma g}(2y_n-x_n), \\
x_{n+1} & =  & x_n+\lambda_n(z_n-y_n).
\end{array}\right.
\end{equation} 
In the case when $f$ and $g$ are convex, this is just the the usual Douglas--Rachford splitting algorithm for the sum convex functions. The iteration \eqref{eq:DR splitting} can be cast as a fixed point iteration in the sequence $(x_n)_{n\in\mathbb{N}}$. Precisely, it may be expressed as
\begin{equation*}
\forall n\in\mathbb{N},\quad x_{n+1} \in (1-\lambda_n)x_n+\lambda_n T_{\rm DR}(x_n),
\end{equation*}
where $T_{\rm DR}$ is the \emph{Douglas--Rachford splitting operator} defined by
\begin{subequations}\label{eq:DRS}
\begin{align}
T_{\rm DR}(x) & :=\frac{1}{2}\left(\Id+(2\prox_{\gamma g}-\Id)\circ(2\prox_{\gamma f}-\Id)\right)(x) \\
&\phantom{:} =\{x+z-y: y\in \prox_{\gamma f}(x), z\in \prox_{\gamma g}(2y-x)\}.
\end{align}
\end{subequations}
In the special case when the functions $f_j$ and $g_i$ are the indicator functions to convex sets, this reduces to the Douglas--Rachford projection algorithms considered in Section~\ref{ssec:projection algorithms}.

We begin by examining properties of the underlying operator $T_{\rm DR}$.
\begin{proposition}[Properties of $T_{\rm DR}$]\label{prop:DRS basic}
Let $f\colon X\to (-\infty,+\infty]$ and $g\colon X\to (-\infty,+\infty]$ be proper functions and let $\gamma >0$. Then the following assertions hold.
\begin{enumerate}[label =(\alph*)]
\item\label{it:DRS Fix} $x\in \Fix T_{\rm DR}$ if and only if there exists $y\in \prox_{\gamma f}(x)$ such that $y\in \prox_{\gamma g}(2y-x)$.
\item\label{it:DRS sFix} $x\in \StrFix T_{\rm DR}$ if and only if $\{y\} =\prox_{\gamma g}(2y-x)$ for all $y\in \prox_{\gamma f}(x)$.
\item\label{it:DRS Fix critical} If $x\in \Fix T_{\rm DR}$, then there exists $y\in \prox_{\gamma f}(x)$ such that $0\in \partial_p f(y) +\partial_p g(y) \subseteq \partial_p (f+g)(y)$. 
\item\label{it:DRS averaged} If $f$ and $g$ are min-convex, then $T_{\rm DR}$ is union $1/2$-averaged nonexpansive.
\end{enumerate}
\end{proposition}
\begin{proof}
\ref{it:DRS Fix} \& \ref{it:DRS sFix}: This follows from \eqref{eq:DRS}.

\ref{it:DRS Fix critical}: Let $x\in \Fix T_{\rm DR}$. Using \ref{it:DRS Fix}, there exists $y\in \prox_{\gamma f}(x)$ such that $y\in \prox_{\gamma g}(2y-x)$. From Proposition~\ref{prop:basic}\ref{it:necessary}, we deduce that
\begin{equation*}
\frac{1}{\gamma}(x-y) \in \partial_p f(y)\text{~~and~~}\frac{1}{\gamma}((2y-x)-y) =\frac{1}{\gamma}(y-x)\in\partial_p g(y).
\end{equation*}
Summing these expressions gives
$0 \in \partial_p f(y) +\partial_p g(y) \subseteq \partial_p (f+g)(y)$,
which proves the claim.

\ref{it:DRS averaged}: By Proposition~\ref{prop:proximity of min}\ref{it:prox f_cvx}, both $\prox_{\gamma f}$ and $\prox_{\gamma g}$ are union $1/2$-averaged nonexpansive. Applying Proposition~\ref{prop:averaged} with $\alpha =1/2$ implies that both $2\prox_{\gamma f}-\Id$ and $2\prox_{\gamma g}-\Id$ are union nonexpansive, and hence so is $(2\prox_{\gamma g}-\Id)\circ (2\prox_{\gamma f}-\Id)$ due to Proposition~\ref{prop:closedness}\ref{it:UAN composition}. Again applying Proposition~\ref{prop:averaged} with $\alpha =1/2$ completes the proof.  
\end{proof}

\begin{proposition}[Finer properties of $T_{\rm DR}$]\label{prop:properties of DRS}
Let $f\colon X\to \mathbb{R}$ be convex, $g :=\min_{i\in I} g_i\colon X\to (-\infty,+\infty]$ be min-convex function, and $\gamma>0$. Then the following assertions hold.	
\begin{enumerate}[label =(\alph*)]
\item\label{it:DRS Fix nec} $x\in\Fix T_{\rm DR}$ if and only if there exists an $i\in I$ such that $y :=\prox_{\gamma f}(x)\in \argmin\{f+g_i\}$ and ${}^\gamma\! g(2y-x) ={}^\gamma\! g_i(2y-x)$. Moreover any such point $y$ satisfies $g(y) =g_i(y)$.
\item\label{it:DRS sFix nec} $x\in\StrFix T_{\rm DR}$ if and only if $y :=\prox_{\gamma f}(x)\in \argmin\{f+g_i\}$ for all $i\in I$ such that ${}^\gamma\! g(2y-x) ={}^\gamma\! g_i(2y-x)$. Moreover, the point $y$ satisfies $g(y) =g_i(y)$.
\item\label{it:DRS sFix local min} If $x\in \StrFix T_{\rm DR}$, then $\prox_{\gamma f}(x)$ is a local minimum of $f+g$.  

\item\label{it:DRS Fix local min} If $x^*\in\StrFix T_{\rm DR}$, then there exists a $\delta>0$ such that 
$$\prox_{\gamma f}\left(  \Fix T_{\rm DR}\cap\mathbb{B}(x^*;\delta)\right)\subseteq\{x\in X:x\text{ is a local minimum of }f+g\}.$$
\end{enumerate}	
\end{proposition}
\begin{proof}

We first note that $\prox_{\gamma f}$ is single-valued by \cite[Proposition~12.15]{BauCom17} and, by  Proposition~\ref{prop:proximity of min}, that 
\begin{equation*}
\prox_{\gamma g}(x) =\{ \prox_{\gamma g_i}(x): i\in I,\ {}^\gamma\! g_i(x) ={}^\gamma\! g(x)\}.
\end{equation*}

\ref{it:DRS Fix nec}:~Let $x\in X$ and denote $y :=\prox_{\gamma f}(x)$. By Proposition~\ref{prop:DRS basic}\ref{it:DRS Fix}, $x\in\Fix T_{\rm DR}$ if and only if
\begin{equation*}
y \in \prox_{\gamma g}(2y-x) =\{\prox_{\gamma g_i}(2y-x): i\in I,\ {}^\gamma\! g(2y-x) ={}^\gamma\! g_i(2y-x)\}.
\end{equation*}
This is equivalent to the existence of an index $i\in I$ such that $y =\prox_{\gamma g_i}(2y-x)$ and ${}^\gamma\! g(2y-x) ={}^\gamma\! g_i(2y-x)$. Since $f$ has full domain, by applying \cite[Corollary~27.3(i)\&(iii)]{BauCom17} to $f$ and $g_i$, we deduce that
\begin{equation*}
y =\prox_{\gamma f}(x) =\prox_{\gamma g_i}(2y-x) \iff y\in \argmin\{f+g_i\}.
\end{equation*}
Moreover, it follows from $y =\prox_{\gamma g_i}(2y-x)\in \prox_{\gamma g}(2y-x)$ and ${}^\gamma\! g(2y-x) ={}^\gamma\! g_i(2y-x)$ that
\begin{equation*}
g(y) ={}^\gamma\! g(2y-x) -\frac{1}{2\gamma}\|y-x\|^2 ={}^\gamma\! g_i(2y-x) -\frac{1}{2\gamma}\|y-x\|^2 =g_i(y), 
\end{equation*}
which completes the claim.

\ref{it:DRS sFix nec}:~Let $x\in X$ and denote $y :=\prox_{\gamma f}(x)$. By Proposition~\ref{prop:DRS basic}\ref{it:DRS sFix}, $x\in\StrFix T_{\rm DR}$ if and only if
\begin{equation*}
\{y\} =\prox_{\gamma g}(2y-x) =\{\prox_{\gamma g_i}(2y-x): i\in I,\ {}^\gamma\! g(2y-x) ={}^\gamma\! g_i(2y-x)\}, 
\end{equation*}
which is equivalent to $y =\prox_{\gamma g_i}(2y-x)$ for all $i\in I$ such that ${}^\gamma\! g(2y-x) ={}^\gamma\! g_i(2y-x)$.  The remainder of the proof is similar to \ref{it:DRS Fix nec}.

\ref{it:DRS sFix local min}:~Let $x\in\StrFix T_{\rm DR}$ and denote $y :=\prox_{\gamma f}(x)$. Then, by \ref{it:DRS sFix nec}, 
\begin{equation}\label{eq:DRS sFix}
 i\in I\text{~~and~~} {}^\gamma\! g(2y-x) ={}^\gamma\! g_i(2y-x) \implies y\in \argmin\{f+g_i\}.
\end{equation}
Now, take any $i\in I$ satisfying $g(y) =g_i(y)$. Noting that $y\in \prox_{\gamma g}(2y-x)$ and using Proposition~\ref{prop:proximity of min}\ref{it:env f}, it follows that
\begin{equation*}
{}^\gamma\! g(2y-x) =g(y) +\frac{1}{2\gamma}\|(2y-x)-y\|^2 
 =g_i(y) +\frac{1}{2\gamma}\|(2y-x)-y\|^2 \geq {}^\gamma\!g_i(2y-x) \geq {}^\gamma\! g(2y-x),
\end{equation*}
which yields ${}^\gamma\! g(2y-x) ={}^\gamma\! g_i(2y-x)$. By \eqref{eq:DRS sFix}, $y\in \argmin\{f+g_i\}$. Since this holds for any $i\in I$ satisfying $g(y) =g_i(y)$, Proposition~\ref{prop:proximity of min}\ref{it:local min} implies that $y$ is a local minimum of $f+g$.

\ref{it:DRS Fix local min}: Let $x^*\in\StrFix T_{\rm DR}$ and set $y^* :=\prox_{\gamma f}(x^*)$. By \ref{it:DRS sFix local min}, there exists $\delta_1>0$ such that $y^*$ is a minimum of $f+g$ on $\mathbb{B}(y^*;\delta_1)$. Consider $\phi\colon X\setto I$ given by 
\begin{equation*}
 \phi(x) :=\{i\in I: {}^\gamma\! g(2y-x) ={}^\gamma\! g_i(2y-x),\ y :=\prox_{\gamma f}(x)\}.
\end{equation*} 
Noting that $x\mapsto 2\prox_{\gamma f}(x)-x$ is continuous, combining
Propositions~\ref{prop:osc_composition} and \ref{prop:continuous rep} shows that $\phi$ is osc. Consequently, Proposition~\ref{prop:prop2.1 Tam17} yields the existence of $\delta_2>0$ such that 
\begin{equation*}
\forall x\in \mathbb{B}(x^*;\delta_2),\quad \phi(x)\subseteq \phi(x^*).
\end{equation*}
Now take $\delta \in (0,\min\{\delta_1, \delta_2\})$, let $\overline{x}\in \Fix T_{\rm DR}\cap \mathbb{B}(x^*;\delta)$, and set $\overline{y} :=\prox_{\gamma f}(\overline{x})$. Using \ref{it:DRS Fix nec}, there exists $i\in \phi(\overline{x})$ such that $\overline{y}\in \argmin\{f+g_i\}$ and $g(\overline{y}) =g_i(\overline{y})$. Since $x^*\in\StrFix T_{\rm DR}$ and $i\in \phi(\overline{x})\subseteq \phi(x^*)$, it also holds that $y^*\in \argmin\{f+g_i\}$ and $g(y^*) =g_i(y^*)$ due to \ref{it:DRS sFix nec}. We deduce that 
\begin{equation*}
f(\overline{y})+g(\overline{y}) =f(\overline{y})+g_i(\overline{y}) =f(y^*)+g_i(y^*) =f(y^*)+g(y^*).
\end{equation*} 
Moreover, $\overline{y}\in \mathbb{B}(y^*;\delta)\subset\inte\mathbb{B}(y^*;\delta_1)$ since $\overline{x}\in B(x^*;\delta)$ and $\prox_{\gamma f}$ is nonexpansive. Therefore, $\overline{y}$ is also a local minimum of $f+g$.
\end{proof}

\begin{theorem}[Douglas--Rachford splitting] 
Let $f\colon X\to \mathbb{R}$ be convex, $g\colon X\to (-\infty,+\infty]$ be min-convex, and $\gamma>0$. Suppose that $x^*\in\StrFix T_{\rm DR}$ and let $(\lambda_n)_{n\in\mathbb{N}}$ be a sequence in $(0,2]$ with $\liminf_{n\to\infty}\lambda_n(2-\lambda_n)>0$. Denote $r :=r(x^*;T_{\rm DR})\in(0,+\infty]$ and consider a sequence $(x_n)_{n\in\mathbb{N}}$ given by \eqref{eq:DR splitting} with $x_0\in\inte\mathbb{B}(x^*;r)$. Then $(x_n)_{n\in\mathbb{R}}$ converges to a point $\overline{x}\in\Fix T_{\rm DR}\cap \mathbb{B}(x^*;r)$. 
Furthermore, for sufficiently small $r>0$, $\prox_{\gamma f}(\overline{x})$ is a local minimum of $f+g$.
\end{theorem}
\begin{proof}
Convergence of the sequence $(x_n)_{n\in\mathbb{N}}$ to a fixed point, $\overline{x}$, follows by combining Proposition~\ref{prop:DRS basic}\ref{it:DRS averaged} and Corollary~\ref{cor:union averaged}. The remaining conclusion follows from Proposition~\ref{prop:properties of DRS}\ref{it:DRS Fix local min}.
\end{proof}

\section{Conclusions}
In this work, we have introduced and studied the classes of union nonexpansive and union averaged nonexpansive operators. Fixed point iterations based on these operators are locally convergent to fixed points when initialized near strong fixed points. The convergence behavior of proximal algorithms applied to minimization problems involving min-convex functions can be systematically studied using the notion.

\subsection*{Acknowledgments}
MND was partially supported by the Australian Research Council (ARC) Discovery Project DP160101537 and by the Priority Research Centre for Computer-Assisted Research Mathematics and its Applications (CARMA) at the University of Newcastle. 
He wishes to acknowledge the hospitality and the support of D. Russell Luke during his visit to the Universit\"at G\"ottingen. MKT was partially supported by a Postdoctoral Fellowship from the Alexander von Humboldt Foundation.

\end{document}